\documentclass[english,a4paper,12pt]{article}
\usepackage{a4}
\usepackage{amsmath}
\usepackage{amsthm}
\usepackage{latexsym}
\usepackage{amssymb}
\usepackage{babel}
\usepackage[latin1]{inputenc}
\usepackage{fancyhdr,array}
\usepackage[pdftex]{graphicx}
\newcommand{\1}{\hbox{ 1 \hskip -7pt I}}

\allowdisplaybreaks
\usepackage{setspace}
\usepackage{indentfirst}
\newtheorem{definition}{Definition}[section]
\newtheorem{proposition}{Proposition}[section]
\newtheorem{lemma}{Lemma}[section]
\newtheorem{theorem}{Theorem}[section]
\newtheorem{corollary}{Corollary}[section]

\newtheorem{question}{Question}[section]

\title{Reinforcement learning in signaling game}
\author{Yilei Hu$^1$, Brian Skyrms$^2$ and Pierre Tarr\`es$^3$}

\newcommand{\Cc}{\mathcal{C}}
\newcommand{\Gc}{\mathcal{G}}
\newcommand{\Sc}{\mathcal{S}}

\newtheorem{corol}{Corollary}
\newtheorem{rema}{Remark}
\newtheorem{exa}{Example}

\renewcommand{\le}{\leqslant}

\renewcommand{\ge}{\geqslant}

\newcommand{\bal}{\begin{align*}}
\newcommand{\eal}{\end{align*}}
\newcommand{\beq}{\begin{eqnarray*}}
\newcommand{\eeq}{\end{eqnarray*}}
\newcommand{\bte}{\begin{theorem}}
\newcommand{\ete}{\end{theorem}}
\newcommand{\bl}{\begin{lemma}}
\newcommand{\el}{\end{lemma}}
\newcommand{\bd}{\begin{description}}
\newcommand{\ed}{\end{description}}

\newcommand{\bc}{\begin{cases}}
\newcommand{\ec}{\end{cases}}
\newcommand{\bp}{\begin{proof}}
\newcommand{\ep}{\end{proof}}
\newcommand{\bco}{\begin{corol}}
\newcommand{\eco}{\end{corol}}

\newcommand{\iy}{\infty}
\newcommand{\tx}{\text}
\newcommand{\R}{\ensuremath{\mathbb{R}}}

\newcommand{\N}{\ensuremath{\mathbb{N}}}

\newcommand{\hV}{\hat{V}}

\newcommand{\Cst}{\mathsf{Cst}}

\newcommand{\F}{\mathcal{F}}

\newcommand{\Es}{\mathbb{E}}
\newcommand{\Pb}{\mathbb{P}}

\newcommand{\Ff}{\mathbb{F}}
\newcommand{\Gf}{\mathbb{G}}

\newcommand{\Nc}{\mathcal{N}}

\newcommand{\Gg}{\mathcal{G}}

\newcommand{\g}{\gamma}

\newcommand{\G}{\Gamma}

\newcommand{\De}{\Delta}

\newcommand{\e}{\epsilon}

\newcommand{\tet}{\theta}

\newcommand{\la}{\lambda}
\newcommand{\La}{\Lambda}

\newcommand{\s}{\sigma}

\def\bdes{\begin{description}}
\def\edes{\end{description}}

\begin{document}
\maketitle
\footnotetext[1]{University of Oxford, Mathematical Institute, 
24-29 St Giles, Oxford OX1 3LB, United Kingdom. E-mail: {\tt Yilei.Hu@maths.ox.ac.uk.}}
\footnotetext[2]{School of Social Sciences, University of California at Irvine, CA 92607. E-mail: {\tt bskyrms@uci.edu}}
\footnotetext[3]{CNRS, Universit\'e de Toulouse, Institut de Math\'ematiques,
118 route de Narbonne, 31062 Toulouse Cedex 9, France. On leave from the Mathematical Institute, University of Oxford.
E-mail: {\tt tarres@math.univ-toulouse.fr}}
{\abstract{We consider a signaling game originally introduced by Skyrms, which models 
how two interacting players learn to signal each other and thus create a common language. 
The first rigorous analysis was done by Argiento, Pemantle,
Skyrms and Volkov (2009) with 2 states, 2 signals and 2 acts. 
We study the case of $M_1$ states, $M_2$ signals and $M_1$ acts for general $M_1$, $M_2$ $\in\N$. 
We prove that the expected payoff increases in average and thus converges a.s., and that a limit bipartite graph emerges, such that no signal-state correspondence is associated to both a synonym and an informational bottleneck. 
Finally, we show that any graph correspondence with the above property is a limit configuration with positive probability.}}

\section{Introduction}

\subsection{Signaling game}
\label{sec:sg}
Signaling games aim to provide a theoretical framework for the following basic question: how do individuals create a common language?
The setting was introduced as follows  by Philosopher D. Lewis (1969).
\begin{description}
\item Consider two players, one called Sender, who is regularly given some information that the other does not have and seeks to transmit it,  and the other  called Receiver. Sender has a fixed
set of signals at his disposal throughout the game, which do not
have any intrinsic meaning  at the very beginning, in the sense that no signal is {\it a priori } associated to any
state; similarly Receiver has a fixed  set of
possible acts. The game is thereafter repeatedly
played according to following steps: (1) Sender observes a certain
state of nature, of which Receiver is not aware. (2) Sender chooses
a signal and then sends it to Receiver. (3) Receiver observes the
signal but not the state, and then chooses an act. (4) Both
players receive payoffs at the end of each round, which are
functions of state and act.
\end{description}
The process involving the above four steps is called one
\emph{communication}. Sender, Receiver, states, signals and acts
constitute one basic \emph{communication system}. The game lies in the
choices of signals and acts by the agents. Note that we do not
fix any strategy at this point; this is the purpose of Section
\ref{Reinlea}.

\subsection{Mathematical definition}\label{mathdef}
As we will explain at length later on, we adopt a dynamic perspective to investigate this game. The outcome of once-play game is not of  much interest to us. What we wonder is whether player can establish a signaling system, within which each signal is uniquely related to a certain state, if they play this game repeatedly.

In this section, we provide solid mathematical definitions for the objects appearing in  Section \ref{sec:sg} and we end with a mathematical definition of repeated signaling game.

\noindent \textbf{Probability space}. Let $(\Omega, \mathcal{F},\mathbb{P})$ be a probability space, which is a sufficiently rich source of randomness. More specifically, the probability space is at least rich enough for all the random variables appearing in this section to live.

\noindent \textbf{State spaces.} Let $\mathcal{S}_1$ be the set of states,  $\mathcal{S}_2$ be the set of signals and $\mathcal{A}$ be the set of acts.

\noindent\textbf{Players and strategies.} We here introduce Nature as a player in this game who assigns a state of nature to Sender at each round of the game. Three players, Nature, Sender  and Receiver  respectively generate a random sequence, denoting their  strategies throughout this repeated game. More specifically,
\begin{description}
  \item [(1)] Nature generates a sequence $(S_n)_{n\in\mathbb{N}}$ of  random variables taking value in the set of states $\mathcal{S}_1$, each denoting which state Nature assigns to Sender at each round of the game.
  \item[(2)] Sender generates  a sequence $(Y_n)_{n\in \mathbb{N}}$ of random variables taking in value in the set of signals $\mathcal{S}_2$, each denoting which signal Sender sends to Receiver at each round of the game.
  \item[(3)] Receiver generates a sequence $(Z_n)_{n\in \mathbb{N}}$ of random variables taking value in the set of states $\mathcal{A}$, each denoting which state Receiver interpret the signal as at each round of the game.
\end{description}

\noindent \textbf{Payoffs.} Let mappings $u^1_n$ and $u^2_n$ from $\mathcal{S}_1\times\mathcal{S}_2\times \mathcal{A}$ to $[0,\infty)$ be payoff functions for Sender and Receiver at the $n$-th round of the game. In other words, Sender and Receiver gain payoffs $u_n^1(S_n,Y_n,Z_n)$ and $u_n^2(S_n,Y_n,Z_n)$ at the end of $n$-th round of the game respectively.

 \noindent \textbf{Information.} Let  filtration $(\mathcal{F}^1_n)_{n\in\mathbb{N}}$ (resp. $(\mathcal{F}^2_n)_{n\in\mathbb{N}}$) denote the information available to Sender (resp. Receiver) before he makes his decision at each round of the game:  for each $n\in\mathbb{N}$, we let
\begin{align*}
  \mathcal{F}^1_n\,&:=\,\sigma\Big(S_i, 0\le i \le n,  \, Y_i, u_n^1(S_i,Y_i,Z_i), 0\le i\le n-1\Big),\\
  \mathcal{F}^2_n\,&:=\,\sigma\Big(Y_i, 0\le i \le n, \,   Z_i, u_n^2(S_i,Y_i,Z_i), 0\le i\le n-1\Big).
\end{align*}

\noindent \textbf{Initial settings.} The distributions of i.i.d random variables $(S_n)_{n\in\mathbb{N}}$ and the distributions of $Y_0$ and $Z_0$ are given at the beginning of the game, which are called $priori$ distribution.

\noindent\textbf{Updating rule for strategies.} Let $\mathcal{F}^1_{n}$-measurable mapping $\rho^1_n$
\begin{align}
 dist.(Y_{n+1})\,= \,\rho^1_n\Big(S_i, 1\le i \le n,  \, Y_i, u_i^1(S_i,Y_i,Z_i), 1\le i\le n-1\Big)
\end{align}
be updating rule of strategies for Sender at $n$-th round of the game.

Let $\mathcal{F}^2_{n}$-measurable mapping $\rho^2_n$
\begin{align}
 dist.(Z_{n+1})\,= \,\rho^2_n\Big(S_i, 1\le i \le n,  \, Z_i,u_i^2(S_i,Y_i,Z_i), 1\le i\le n-1\Big)\end{align} be updating rule of strategies for Receiver at $n$-th round of the game.  Different $(\rho^1_n)_{n\in\mathbb{N}}$ and $(\rho^2_n)_{n\in\mathbb{N}}$ represent different learning rules.

\begin{definition}[Signaling game]
A repeated signaling game $G$ is defined as: $$G=\Big((\Omega,\mathcal{F},\mathbb{P}), \mathcal{S}_1,\mathcal{S}_2,\mathcal{A}, (S_n,Y_n,Z_n, u^1_n,u^2_n,\rho^1_n,\rho^2_n)_{n\in\mathbb{N}}\Big).$$
\end{definition}

\subsection{Questions}
Throughout the paper, we limit our attention to a special, but most
common, circumstance under which
\begin{description}
\item[\textbf{(A1)}] the set of acts matches the
set of states of nature by a bijective map; and \item[\textbf{(A2)}] both players
only receive fixed payoffs when the act chosen by Receiver matches
the state of nature, which is considered as a \emph{successful
communication}; otherwise they obtain nothing.\end{description}

Under these assumptions, the act can be understood as an interpreted state. Note however that the Receiver does not necessarily know the possible states of nature: the mutual goal for the two players is to make the communication
succeed, but they are not always aware that they actually
coordinate each other, or even that they are involved in a
communication game.

The analysis of choices of strategies by the players
gives rise to the following
\begin{question}
What game-theoretical equilibrium is  most likely to arise in this repeated
game?
\end{question}
This issue can either be analyzed from a static perspective with classic
equilibrium analysis,  see Trapa
 and Nowak (2000), Huttegger (2007) and Pawlowitsch (2008), or in a dynamic perspective, through individual learning models or evolutionary strategies,  see for instance  Huttegger and Zollman (2010). The latter perspective investigates the evolutionary pathway out of equilibria:
\begin{question}
Does the communication system asymptotically reach a stable equilibrium
state? If so, what are the good candidates, and  how does the communication system reach them?
\end{question}

However, modeling the game through individual learning process is really an issue involving lots of factors, for instance the level of rationality of players. In particular, we are interested in the following question which is first raised by Skyrms,
\begin{question}
What is the simplest mechanism to ensure the emergence of a signaling system in this repeated game?
\end{question}

\subsection{Reinforcement learning} \label{Reinlea}
\label{sec:rl}
In this paper we adopt a dynamic perspective, based on the following individual reinforcement learning model.
\\ \vspace{0.005in}\\
 \textbf{(A3) Roth-Erev reinforcement learning rule} (or Herrnstein's matching law):\emph{ the probability of
choosing an action is proportional to its accumulated rewards.}\\
\vspace{0.005in} \\
Assumption \textbf{(A3)} actually decide the updating rule for distributions of strategies for players. The corresponding behavior is analyzed by Argiento et al. (2009) in the 2-state, 2-signal, 2-act case, who show  that an optimal signaling system emerges eventually, in the sense  of a one-to-one correspondence between states and signals.  We study here the case of $M_1$ states, $M_2$ signals and $M_1$ acts for general $M_1$, $M_2$ $\in\N$.

Note that Roth-Erev reinforcement rule is one of many possible strategies of the players, who can have various levels of rationality, each of them leading to a different learning process; for instance the myopic and best response models, see Fudenberg and Levine (1998). Let us  briefly motivate the reinforcement condition, corresponding to a low level of rationality. It is natural to believe that  individuals with high rationality, devoting themselves to the task of establishing a common language, would rapidly succeed. It is interesting to study whether on the contrary, under the only assumption that these individuals have good memory of their own past experience and aspire to a better payoff, a signaling system would also emerge, and how optimal the limiting system is. This pertains either to individuals with lower cognitive ability, or who do not devote themselves totally to the task of learning the game or take optimal decisions.

\subsection{The model}
\label{sec:model}
\subsubsection{Assumptions}

Let $G=\Big((\Omega,\mathcal{F},\mathbb{P}), \mathcal{S}_1,\mathcal{S}_2,\mathcal{A}, (S_n,Y_n,Z_n, u^1_n,u^2_n,\rho^1_n,\rho^2_n)_{n\in\mathbb{N}}\Big)$ be a signaling game as defined in Section \ref{mathdef}. Apart from assumptions \textbf{(A1)}-\textbf{(A3)}, we make further more assumptions for our model. One is about the $priori$ distributions and the other is a more detailed version of assumption $\textbf{(A2)}$.
\begin{description}
  \item[\textbf{(A4)}] States of nature are equiprobable. In other words, for each $n\in\mathbb{N}$, $S_n$ is an  independent uniformly distributed random variable. Furthermore, $Y_0$ and $Z_0$ are independent uniformly distributed random variable.
  \item[\textbf{(A5)}] Apart from assumption \textbf{(A2)} on payoffs, we here assume that payoffs are constants only dependent on types of states and acts. More precisely, $u^1_n(S_n,Y_n,Z_n)=u^2_n(S_n,Y_n,Z_n)=a_{Z_n,Y_n}\1_{\{S_n=Z_n\}}$, where $a_{i,j}$ is a positive constant, for $i\in\mathcal{S}_1, \,j\in\mathcal{S}_2$.
\end{description}

\subsubsection{The model}
Under assumptions $\textbf{(A1)}-\textbf{(A5)}$, we now present the model of the signaling game we are going to study in this paper.

Suppose there are $M_1$ states, $M_2$ signals and $M_1$ acts.  Let
$\mathcal{S}=\mathcal{S}_1\cup \mathcal{S}_2$ and, for all $d\in\N$, let
$\mathcal{S}^d:=\mathcal{S}\times\ldots\times\mathcal{S}$.


Let
$\mathcal{S}_{\textnormal{pair}}:=\{(i,j)\,:\,i\in\mathcal{S}_1,\,j\in
  \mathcal{S}_2\}$ be the set of \emph{strategy pairs}. Note that a strategy
  pair $(i,j)$ carries different meanings for the  Sender and for the Receiver. For the Sender,
  it means choosing signal $j$ when he observes state $i$, while for the Receiver it means choosing act $i$ when he receives signal $j$. This strategy
  pair $(i,j)$ accumulates the same payoffs for both the Sender and
the  Receiver, since the corresponding rewards are always received at the same
  time: let $V(n,i,j)$ denote this accumulated payoff at time  $n$. For each $n\in\mathbb{N}$, let
$$V_n:=\left(V(n,i,j)\right)_{i \in
\mathcal{S}_1,\,j \in \mathcal{S}_2}$$ be the
\emph{payoff vector} at time $n$.

Let us describe the random process $(V_n)_{n\in\N}$, arising from our model for the game and its strategies in Sections \ref{sec:sg}--\ref{sec:rl}:
\begin{itemize}
\item[1] \textbf{ Initial setting.} For any $(i,j)\in\Sc_{\textnormal{pair}}$, we assume that $V(0,i,j)>0$
is fixed.

\item[2] \textbf{ Reinforcement learning.} At each time step, Sender
observes a certain state $i$ from set of states $\mathcal{S}_1$; we assume here that all states arise with equal probability $1/M_1$. Then  Sender randomly chooses a signal, his probability of drawing $j$ being
$$ \frac{V(n,i,j)}{\sum_{l\in\mathcal{S}_2}V(n,i,l)}.$$
Receiver
observes the signal he receives (let us call it $j$) and then randomly chooses an act $k$   with
probability
$$ \frac{V(n,k,j)}{\sum_{l\in\mathcal{S}_1}V(n,l,j)}.$$

\item[3] \textbf{ Updating rule.} Both Sender and Receiver receive payoffs
when the act chosen by Receiver matches the state observed by Sender. For
any $i\in\mathcal{S}_1,\,j\in\mathcal{S}_2$,
\begin{displaymath}
V(n+1,i,j):=\left\{
\begin{array}{ll}
V(n,i,j)+a_{i,j} & \textrm{if Sender observes state $i$ and chooses }\\ & \textrm{
signal $j$, and Receiver chooses act $i$;}\\
V(n,i,j) &\textrm{if else.}
\end{array} \right.
\end{displaymath}
\end{itemize}
We suppose in this paper that $a_{i,j}:=1$ for all $i\in\Sc_1$, $j\in\Sc_2$. We also assume for simplicity that $V(0,i,j)=1$, for all
$i\in\mathcal{S}_1$, $j\in\mathcal{S}_2$, but the proofs carry on to general initial conditions.

\subsubsection{Symmetrization and key processes}\label{SEC:12}
Let us symmetrize the notation, which will simplify some proofs (in particular that of Proposition \ref{Prop:Lya}): for all $i$ $j$ $\in\Sc$, let
\begin{align*} V(n,i,j)&:=\left\{\begin{array}{ll} 0 & \quad i,j \in
\mathcal{S}_1 \textrm{ or } i,j \in \mathcal{S}_2\\
V(n,i,j)& \quad i \in \mathcal{S}_1, \,j
\in \mathcal{S}_2\\
V(n,j,i)& \quad j \in \mathcal{S}_1,\, i
\in \mathcal{S}_2\end{array} \right.
\end{align*}

Now, for all $n\in\N$ and  $i\in\Sc$ state or signal, let
\begin{align*}
T_n^i:=\sum_{j \in \mathcal{S}}V(n,i,j)
\end{align*}
be its number of successes up to time $n$.

For all $n\in\N$, let
$$T_n:=\frac{1}{2}\sum_{i \in \mathcal{S}}T_n^i=\sum_{i \in
\mathcal{S}_1}T_n^i=\sum_{i \in
\mathcal{S}_2}T_n^i.$$
Then $T_n-T_0$ is the total number of successes of the communication system up to time $n$.

Let
\begin{align*}  x_{ij}^n&:= V(n,i,j)/T_{n}
\quad i,j \in
\mathcal{S},\, n \in \mathbb{N}; \\
x_i^n&:=T_n^i/T_n, \quad i \in \mathcal{S},\, n \in \mathbb{N}.
\end{align*}

For all $n\in\N$, let
$$x_n:=\left(x_{ij}^n\right)_{i,j \in
\mathcal{S}},\,n\in\mathbb{N}$$
be the \emph{occupation measure }at time $n$, which takes values in the interior
of the simplex
$$ \Delta:=\big\{(x_{ij})_{i,j\in \mathcal{S}}:
x_{ij}\ge 0, \sum_{i\in \mathcal{S}_1,~j\in \mathcal{S}_2}x_{ij}=1, x_{ji}=x_{ij}\tx{ for all }i,j\in\Sc\big\}.$$

Let us define
$$\partial\Delta:=\{\,x\in\Delta\,:\,\exists \,i\in\mathcal{S},\,s.t.\,x_i=0\,\},$$
which we call \emph{boundary} throughout the paper, although it is
not the topological boundary of $\Delta$. One of the technical difficulties in this model is the understanding of the behavior of $(x_n)_{n\in\mathbb{N}}$ near the boundary, as we shall explain later.

Given  $x\in\Delta\setminus\partial\Delta$, $i$, $j$ $\in\Sc$, let
$$y_{ij}:=\frac{x_{ij}}{x_ix_j}$$ be the \emph{efficiency}  of the strategy
pair $(i,j)$ and let
$$N_i(x):=\sum_{k\in\mathcal{S}}\frac{x_{ik}}{x_i}\cdot y_{ik}$$
be the \emph{efficiency $N_i(x)$ of
$i$.} We will justify this notation in Section \ref{sec:ode}.

Note that processes $(x_n)_{n\in\mathbb{N}}$ and $(T_n)_{n\in\mathbb{N}}$ contain all the important information of the communication system throughout the game. Therefore, to study how our model evolves, we only need to focus on the evolutions of $(x_n)_{n\in\mathbb{N}}$ and $(T_n)_{n\in\mathbb{N}}$.

\subsection{Urn setting: Another way to interpret the model}
Note that the reinforcement learning and updating rules 2 and 3 in Section \ref{sec:model} can be interpreted in an urn setting. Assume Sender has $M_1$ urns indexed by states, each of them having $M_2$ colours of balls, one per signal. Similarly assume that Receiver has $M_2$ urns indexed by signals, each of them having $M_1$ colours of balls, one per act (or state, since the two sets coincide).

The model corresponds to the following: Sender picks a ball at random in the urn indexed by the state he observes, and sends the signal given by its colour. Then receiver picks a ball at random in the urn indexed by this signal, and chooses the act given by its colour. Both Sender and Receiver put back the balls they picked and, if the act matches the state, add one more ball of the same colour.

\section{Main results}

Given $x\in\De$ distribution of strategy pairs of
the communication system, we introduce in Section \ref{sec:bgcp} the bipartite graph of state/signal connections associated to $x$, and its communication potential or efficiency, which measures the corresponding expected payoff up to a multiplicative constant. We present in Section \ref{sec:mr} the main results of the paper.

\subsection{Bipartite graph and communication potential}
\label{sec:bgcp}

 \begin{definition} \label{graph.1}
Given $x\in \Delta$, let $\Gc_x$ be the weighted bipartite graph  with vertices
$\mathcal{S}:=\mathcal{S}_1\cup\mathcal{S}_2$, adjacency $\sim$ and
weights as follows \begin{description}
\item[(1)] $\forall \,i \in \mathcal{S}_1, j \in
\mathcal{S}_2, i\sim j$ if and only if $x_{ij}>0.$ \item[(2)] The
weight of edge $\{i,j\}$ is its efficiency $y_{ij}=x_{ij}/(x_ix_j)$.
\end{description}
\end{definition}
Note that the new adjacency relation is only justified when $x$ is in the topological boundary of $\De$; otherwise,
$\Gc_x$ is then the complete $2$-partite graph with partitions $\Sc_1$ and $\Sc_2$.

\begin{definition}
Let $H:\De\longrightarrow\R_+$ be the function defined by, for all $x\in\De$,
$$H(x):=\sum_{i\in \mathcal{S}_1, \, j\in
\mathcal{S}_2:x_{ij}>0}\frac{x_{ij}^2}{x_ix_j}=\frac{1}{2}\sum_{i, j\in
\mathcal{S}:x_{ij}>0}\frac{x_{ij}^2}{x_ix_j}.$$
We call $H(x)$ communication potential or efficiency of $x$.
\end{definition}

Note the {\it communication potential} of $x_n$ at time $n$ can be interpreted - up to a multiplicative constant- as the expected payoff at that time step:
\begin{align}\label{CommPotential} \mathbb{P}(T_{n+1}-T_n=1\,|\,\mathcal{F}_n)
=\frac{1}{M_1}H(x_n),\end{align}
where $\Ff=(\F_n)_{n\in\N}$ is the filtration of the past, i.e. $\F_n:=\s(x_1,\ldots,x_n)$.

\begin{lemma}
$H$ has minimum $1$ and maximum $\min(M_1,M_2)$ on $\Delta$.
\end{lemma}
\begin{proof} Using Cauchy-Schwartz inequality,
\begin{align*}
H(x)=\sum_{i\in\Sc_1, j\in\Sc_2:x_{ij}>0} \frac{x_{ij}^2}{x_ix_j}\sum_{i\in\Sc_1, j\in\Sc_2} x_ix_j&\ge
\left(\sum_{i\in\Sc_1, j\in\Sc_2:x_{ij}>0}
\frac{x_{ij}}{\sqrt{x_ix_j}}\sqrt{x_ix_j}\right)^2\\
&~~~~~~~~=\left(\sum_{i\in\Sc_1, j\in\Sc_2}
x_{ij}\right)^2 = 1
\end{align*}
 provides the first inequality, whereas the
second one comes from
\begin{align*}
H(x)=\sum_{i\in\Sc_1, j\in\Sc_2:x_{ij}>0} \frac{x_{ij}}{x_i}\frac{x_{ij}}{x_j}
&\le\sup_{i\in\Sc_1, j\in\Sc_2:x_{ij}>0} \frac{x_{ij}}{x_i}\sum_{i\in\Sc_1, j\in\Sc_2:x_{ij}>0}\frac{x_{ij}}{x_j} \\
&\le
\sum_{i\in\Sc_1, j\in\Sc_2:x_{ij}>0}\frac{x_{ij}}{x_j}=M_2;
\end{align*}
similarly $H(x)\le M_1$.

Now
\begin{itemize}
  \item[(a)] $H(x)$ reaches the minimum 1 if and only
  if $\mathcal{G}_x$ is a complete graph on which every edge shares
  the same weight $1$, as displayed on the figure below.

  \begin{center}
  \includegraphics[width=0.55\textwidth]{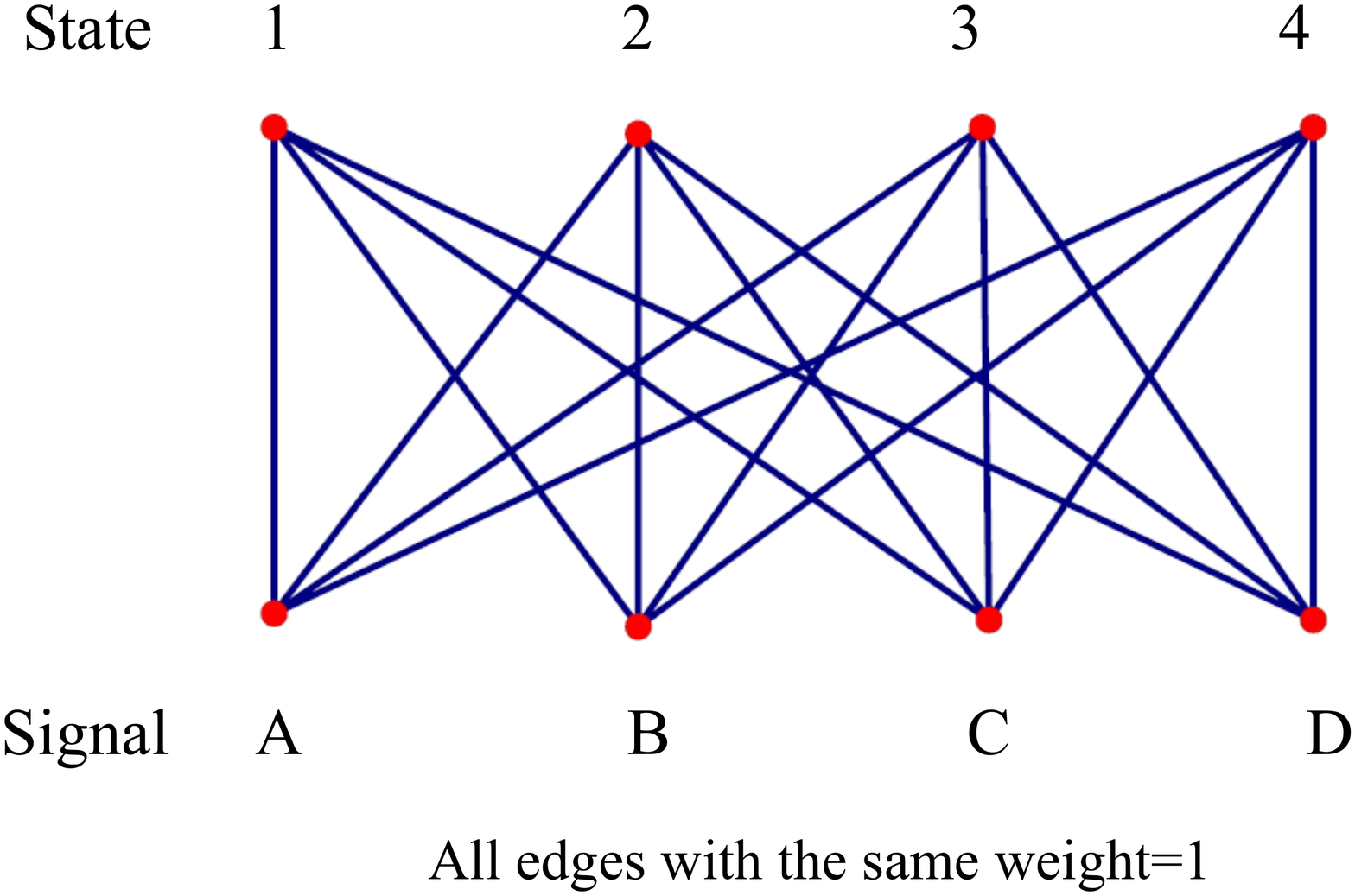}
  \end{center}

\item[(b)] If $M_1\ge M_2$ (resp. $M_1\le M_2$), then $H(x)$ reaches the maximum if and only if every vertex $i\in\Sc_1$ (resp. $\Sc_2$) only has one adjacent edge in $\Gc_x$. In the case $M_1\le M_2$, this corresponds to a unique meaning for every signal, i.e. perfect efficiency, as displayed on the figure below.
  \begin{center}\includegraphics[width=0.50\textwidth]{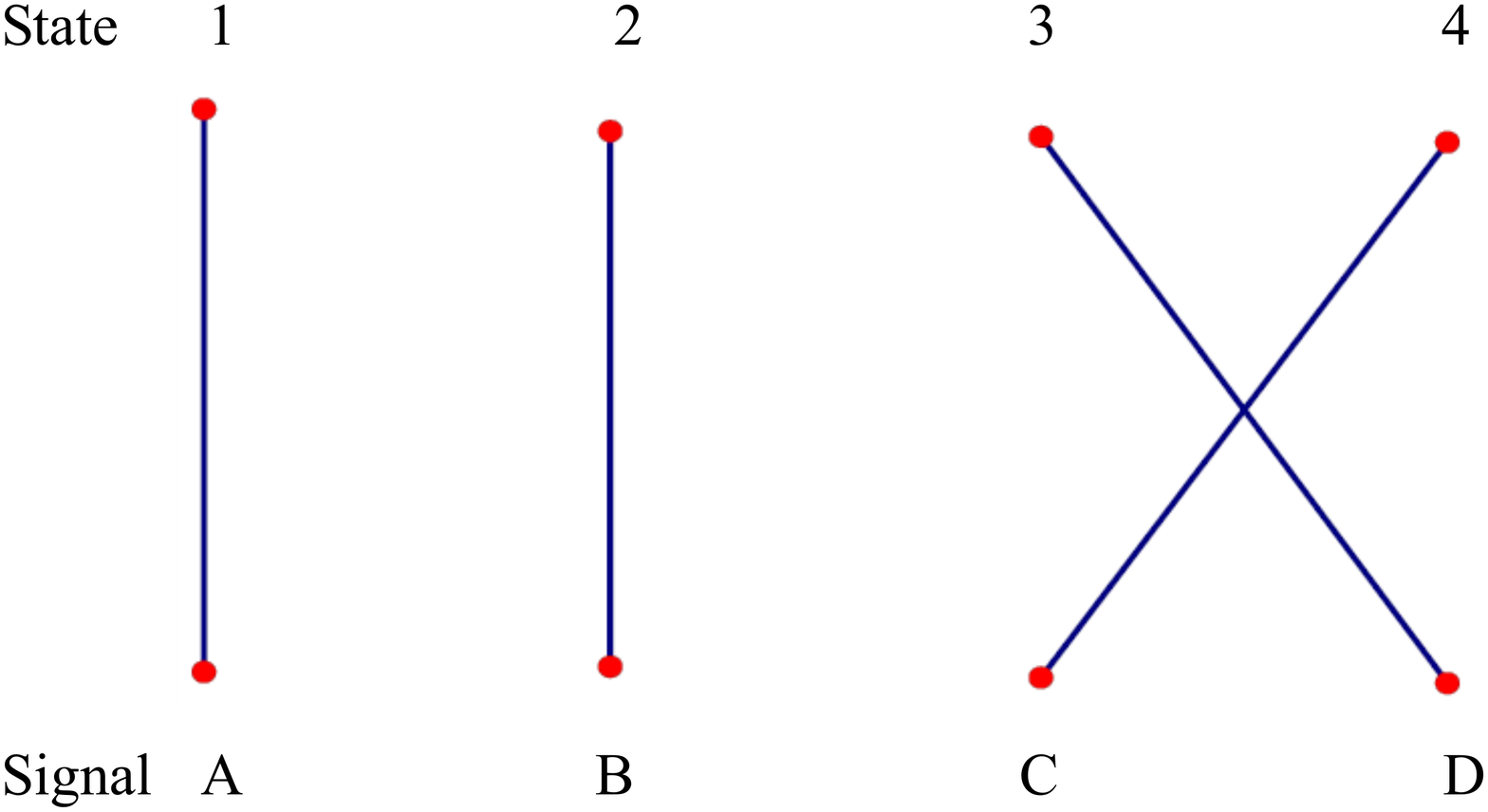}\end{center}
\end{itemize}
\end{proof}

\subsection{Main results}
\label{sec:mr}
 \begin{definition} \label{property}Given a graph $\mathcal{G}$ on
$\mathcal{S}_1\cup \mathcal{S}_2$, let $(P)_\mathcal{G}$ be the
following property:
\begin{itemize}
 \item if we let $\Cc_1$, $\ldots$ $\Cc_d$ be its connected components
then, for every $i\in\{1,\ldots,d\}$, $\Cc_i\cap\Sc_1$ or
$\Cc_i\cap\Sc_2$ is a singleton.
\item each vertex has a corresponding edge.
\end{itemize}
\end{definition}
We call \emph{synonym} (resp.\emph{ informational bottleneck} or \emph{polysemy}) a state (resp. signal)  associated to several signals (resp. states or acts), or the corresponding set of adjacent signals (resp. states). Obviously $M_1\ne M_2$ ensures the existence of at least one synonym or polysemy.

Note that, given $x\in\De$, and even if $M_1=M_2$, property  $(P)_{\Gc_x}$ allows for synonyms or informational bottlenecks, and does not ensure that the system is optimal as a communication system, i.e. that $H(x)$ reaches the maximum of $H$. Most common languages  have such flaws. We show, on the figure below, a graph $\Gc$ corresponding to a sub-optimal communication system: it is easy to check that its normalized efficiency is 80$\%$, i.e. that any $x\in\De$ such that $\Gc_x=\Gc$ is such that $H(x)/\max H=0.8$.
\begin{center}\includegraphics
[width=0.55\textwidth]{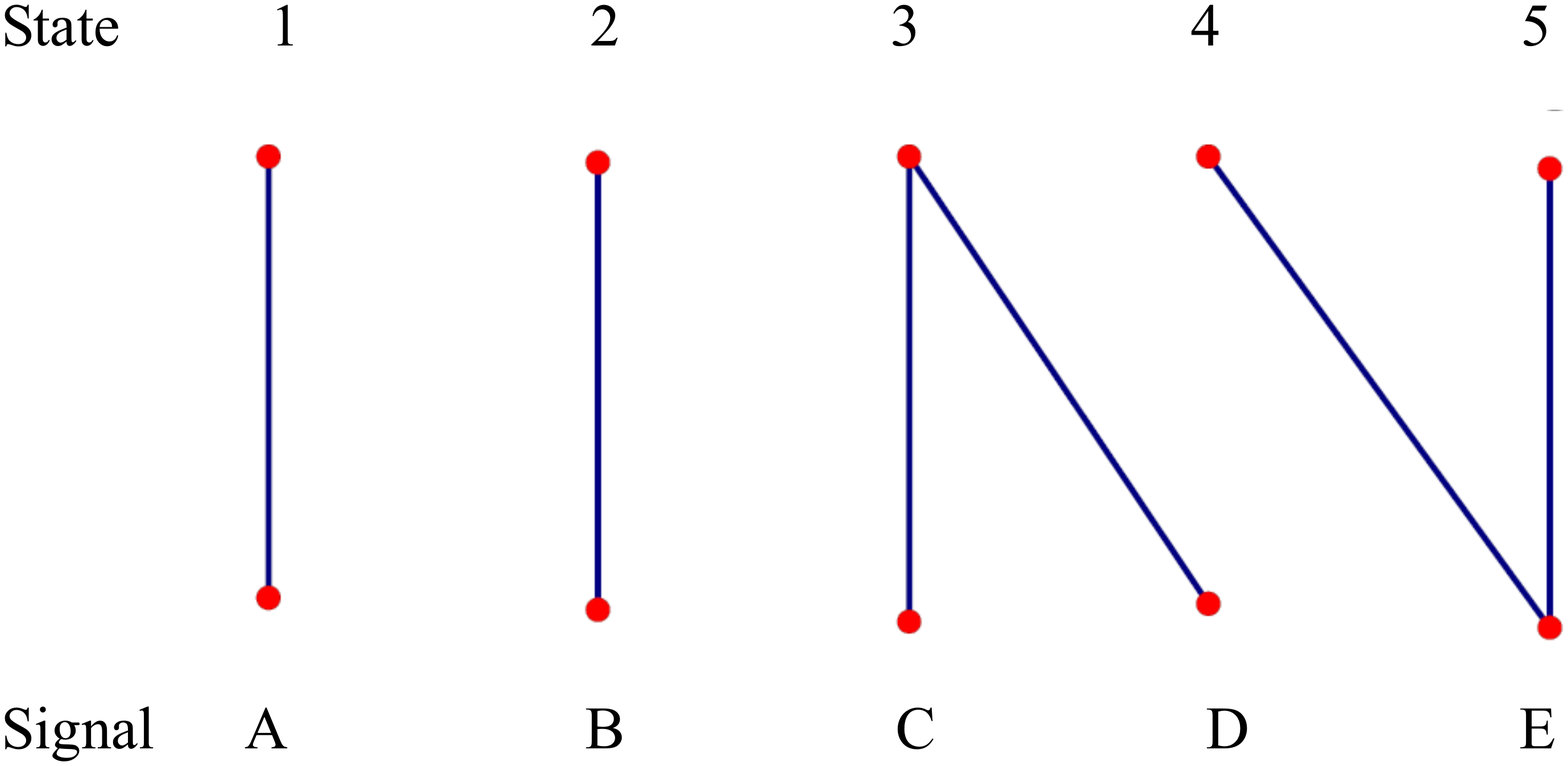}\end{center}

\begin{theorem}\label{theorem3}
 The communication potential process $(H(x_n))_{n\in\N}$ is a bounded submartingale, and hence converges a.s.
\end{theorem}

\begin{theorem}\label{theorem5}
  $(x_n)_{n\in\mathbb{N}}$ converges to the set of equilibria of mean-field ODE almost surely.
\end{theorem}
\noindent \textbf{Remark.} We will define the mean-field ODE in Section \ref{SEC:4:2}. Roughly speaking, the mean-field ODE is derived from the dynamics of expected movement of $(x_n)_{n\in\mathbb{N}}$.
 
\begin{theorem}
\label{theorem2} For all $\mathcal{G}$ on $\mathcal{S}_1\cup
\mathcal{S}_2$ s.t. $(P)_{\mathcal{G}}$ holds, with positive
probability \vspace{-0.10in}\begin{description}
 \item[(a)] $x_n\to x$ s.t. $\mathcal{G}_x=\mathcal{G}$.
\item[(b)] $\forall i,j\in\Sc$, $V(\infty,i,j)=\infty\,\iff\,\{i,j\}$ is an edge of $\mathcal{G}$.
\end{description}
\end{theorem}

\subsection{Contents}

The rest of this chapter is devoted to the proof of the main results, as follows. In Section \ref{SEC:4:1} we discuss the stochastic approximation of
$(x_n)_{n\in\mathbb{N}}$ by an ordinary differential equation. In Section \ref{SEC:Lya} we justify some  notation from Section \ref{sec:model},  and show Theorem \ref{theorem3} and its deterministic counterpart that $H$ is a Lyapounov function of the associated ODE. In Section \ref{SEC:4:5} we deduce that $(x_n)_{n\in\mathbb{N}}$ almost surely converges to the set of equilibria of this differential equation, and describe the stable equilibria $x$ in terms of graph structure of $\Gc_x$. In Section \ref{SEC:4:11} we connect our analysis of stability of the reinforcement learning model with the one from the static equilibrium analysis literature. Finally in Section \label{sec:>0}, we show   Theorem \ref{theorem2} about convergence with positive probability towards subgraphs $\Gc$ satisfying $(P)_{\Gc}$.


\subsection{Notation}
For all $u$, $v$ $\in\R$, we write $u=\Box(v)$ if $|u|\le v$; we let $u\wedge v=\min(u,v)$ (resp. $u\vee v=\max(u,v)$) be the minimum (resp. maximum) of $u$ and $v$.

We let  $\Cst(a_1,a_2,\ldots, a_p)$ denote a positive constant depending
only on  $a_1$, $a_2$, $\ldots$ $a_p$, and let $\Cst$  denote a  universal
positive constant.
\section{Stochastic approximation}
\label{SEC:4:1} \subsection{Mean-field ODE}\label{SEC:4:2}
Let $|\cdot|$ be the Euclidean norm on $\mathbb{R}^{M_1\times
M_2}$. Let us calculate the increment of $(x_n)_{n\in\mathbb{N}}$ at
time $n$:
\begin{align}
\label{s-est1}
x_{n+1}-x_n&=\left(\frac{V_{n+1}}{1+T_n}-\frac{V_n}{1+T_n}+\frac{V_n}
{1+T_n}-\frac{V_n}{T_n}\right)\1_{\{V_{n+1}-V_n\ne0\}}
\nonumber\\
&=\frac{\1_{\{V_{n+1}-V_n\ne0\}}}{1+T_n}(V_{n+1}-V_n-x_n).\end{align}
In expectation,
\begin{align}
\label{expectation}
\mathbb{E}[\,x_{n+1}-x_n\,|\,\mathcal{F}_n\,]=\frac{1}{(1+T_n)M_1}F(x_n),\end{align}
where $F$, defined from
$\Delta$ to $T\De$ tangent space of $\De$, maps $x$ to
\begin{displaymath}
F(x):=\left(x_{ij}\Big(\frac{x_{ij}}{x_ix_j}-H(x)\Big)\right)_{i,j \in
\mathcal{S}},
\end{displaymath}
with the convention that $F(x)_{ij}=0$ if $x_{ij}=0$.

Let us consider the following ordinary differential equation, defined on $\De\setminus\partial\De$:
\begin{align}
\label{ODE} \frac{d\,x}{dt}\,=\,F(x).
\end{align}

\subsection{Results}
\label{SEC:result}
\begin{lemma}\label{lemma:StochApp}
 There exists an adapted martingale increment process $(\eta_n)_{n\in \mathbb{N}}$ such that, for all $n\in\N$,
\begin{align}\label{Stochapprox}x_{n+1}-x_n=\frac{1}{(1+T_n)M_1}F(x_n)+\eta_{n+1},\end{align}
and $|\eta_{n+1}|\le2/(1+T_n)$.
\end{lemma}
\bp
The inequality $|\eta_{n+1}|\le2/(1+T_n)$ comes from
\begin{align}
\label{s-est2}
|x_{n+1}-x_n|=\left|\frac{\1_{\{V_{n+1}-V_n\ne0\}}}{1+T_n}
(V_{n+1}-V_n-x_n)\right|\le\frac{1}{1+T_n}.
\end{align}
\ep

The following Lemma \ref{basic:est}, which proves asymptotic linear growth of $T_n$, will in particular imply that the martingale
$(\sum_{k=1}^n\eta_{k})_{n\in\mathbb{N}}$ converges a.s., by Doob convergence theorem.
\begin{lemma}
\label{basic:est} With Probability 1,$$ \frac{1}{M_1} \le \liminf_{n
\to \infty}\frac{T_n}{n}\le \limsup_{n \to \infty}\frac{T_n}{n}\le
\frac{\min(M_1,M_2)}{M_1}. $$
\end{lemma}
\begin{proof} The result is a direct consequence of  (\ref{CommPotential}) and Conditional Borel-Cantelli Lemma, see \cite[Theorem I.6]{Dur04}
\end{proof}
\noindent\textbf{Remark.}
\textnormal{We show later that, as $n$ goes to infinity, $H(x_n)$ converges (Theorem \ref{theorem3}), and the proof above implies that its limit is also the one of $T_n/n$.}

Formula ($\ref{Stochapprox}$) can be interpreted as a stochastically perturbed Cauchy-Euler
approximation scheme for the ODE
($\ref{ODE}$), with step size $1/M_1(1+T_n)$. The  step size being  $O(1/n)$,
$(x_n)_{n\in\mathbb{N}}$ asymptotically shadows  solutions of  the asymptotic ODE, so that its limit set belongs to a class of possible limit sets of pseudotrajectories of the ODE (see for instance \cite{Ben99}).

Let $\G$ be the \emph{set of
equilibria}
 of the ODE, i.e.
$$\Gamma:=\Big\{x\in\Delta\,:\,
F(x)=0\Big\}.$$

\section{Lyapunov function}
\label{SEC:Lya}
\subsection{Deterministic case: Mean-field ODE}
\label{sec:ode}
Let us start with a heuristical justification of the fact that $H$ is a Lyapounov function, i.e. that it increases along the trajectories of the ODE.

The differential equation \eqref{ODE} can be understood in the language of (non-linear) replicator dynamics, with the following biological perspective. Suppose a population consists of species $(i,j)\in\Sc_1\times\Sc_2$,
each corresponding to a strategy pair. Let the \emph{fitness} of $(i,j)$ be its efficiency $y_{ij}=x_{ij}/(x_ix_j)$. The wording is justified by the following interpretation: the probability that population $(i,j)$ increases is
$$\frac{1}{M_1}\cdot\frac{x_{ij}}{x_i}\cdot\frac{x_{ij}}{x_j}
=\frac{1}{M_1}\times\textnormal{ proportion of } (i,j)\, \times
\textnormal{ its efficiency}.$$ Then the average fitness of the
whole population is the communication potential $H(x)$.

Therefore the mean-field ODE can be understood as follows,
\begin{align*} \qquad &\textrm{\emph{Growth rate of} }x_{ij}\,\\ &\qquad=\, x_{ij}
\,\times\,(\textrm{
 \emph{fitness of}  }(i,j)-\textrm{ \emph{average fitness of whole population}}).
\end{align*}
In particular, those species (i.e. strategy pairs) whose
fitness (i.e. efficiency) are above the average increase their
proportions in the population (i.e. distribution of strategy pairs).
Note that the fitness of our species changes over time.

This interpretation makes it reasonable to conjecture that the average fitness of the whole population indeed
increases along solutions of ODE $(\ref{ODE})$, as we show in Proposition \ref{Prop:Lya} and, in the (discrete) stochastic case, in Theorem \ref{theorem3}, proved in Section \ref{sec:stoc}. The proof of the latter is technically quite long, owing to the non-continuity of $H$ on the boundary $\partial\De$, which prevents us from converting the deterministic Proposition \ref{Prop:Lya} into a stochastic one via a simple Taylor formula.

Let, for all $x\in\De$,
$$p(x):=\frac{1}{2}\sum_{i,j,k\in\mathcal{S}: x_{ij},x_{ik}\ne0}
\frac{x_{ij}x_{ik}}{x_i}\left(y_{ij}-y_{ik}\right)^2.$$
\begin{proposition}
\label{Prop:Lya} $H$ is a Lyapunov function on
$\Delta\setminus\partial\Delta$ for the mean-field ODE
\textnormal{(\ref{ODE})}; more precisely,
\begin{equation}
\label{eq:varh1}
\nabla H \cdot F \,(x)\, =\, p(x)\ge0.
\end{equation}
\end{proposition}
\noindent\textbf{Remark.} \textnormal{ We will see later that $H$ is not a strict
Lyapunov function; in other words, $\nabla H\cdot F$ does not only
vanish on the set of rest points of $F$.}

\begin{proof}
We take advantage of the symmetrical notation introduced at the end of Section \ref{sec:model}: in a
mathematical perspective, there is no difference between
state and signal in the mean-field ODE.
Let us now differentiate $H$ along a path of the ODE:  note that, when differentiating with respect to space variables, we are obviously not restricted to $\De$, so that $x_{ij}$ and $x_{ji}$ are considered as independent variables (in the calculation below, we use the convention that  $x_i=\sum_{j\in\Sc} x_{ij}$, but any other convention would lead to the same result):
\begin{align}\nonumber
\nabla H \cdot F \,(x)\, =\,&\sum_{i,j\in
\mathcal{S}}\frac{ x_{ij}}{x_ix_j}\Big(\frac{ (x_{ij})^2}{x_ix_j}-
x_{ij}H (x) \Big)\\ \nonumber&\qquad -\frac{
(x_{ij})^2}{2(x_i)^2x_j}\Big(\sum_{k \in \mathcal{S}}\frac{
(x_{ik})^2}{x_ix_k}- x_{ik}H (x) \Big)\\ \nonumber&\qquad
-\frac{ (x_{ij})^2}{2x_i(x_j)^2}\Big(\sum_{k \in
\mathcal{S}}\frac{(x_{jk})^2}{x_jx_k}-x_{jk}H (x) \Big)\\
\label{lya1} \,=\,&\sum_{i,j \in
\mathcal{S}}\frac{(x_{ij})^3}{(x_i)^{2}(x_j)^2}-\sum_{i,j,k \in
\mathcal{S}}\frac{ (x_{ij})^2 (x_{ik})^2}{(x_i)^3x_jx_k}\\
\label{lya2}&\qquad -\Big(\sum_{i,j\in
\mathcal{S}}\frac{(x_{ij})^2}{x_ix_j}-\sum_{i,j,k \in
\mathcal{S}}\frac{  (x_{ij})^2x_{ik}}{(x_i)^2x_j} \Big)H (x)
\end{align} Using that $\sum_{k\in
\mathcal{S}}x_{ik}/x_i=1$,
$$
(\ref{lya2})\,=\,\sum_{i,j\in \mathcal{S}}\frac{
(x_{ij})^2}{x_ix_j}\Big(1-\sum_{k\in \mathcal{S}}\frac{
x_{ik}}{x_i}\Big)H (x)\,=\,0.$$ Using the symmetry between $j$ and $k$, we obtain
\begin{align}
\nonumber (\ref{lya1})
&\,=\,\sum_{i,j,k\in\mathcal{S}}\frac{(x_{ij})^3x_{ik}}{(x_i)^3(x_j)^2}-\sum_{i,j,k\in\mathcal{S}}
\frac{(x_{ij})^2(x_{ik})^2}{(x_i)^3x_jx_k}\\ \label{lya3} &
\,=\,\frac{1}{2}\sum_{i,j,k\in\mathcal{S}}
\frac{x_{ij}x_{ik}}{x_i}\left(\frac{x_{ij}}{x_ix_j}-\frac{x_{ik}}{x_ix_k}\right)^2.
\end{align}
\end{proof}

\begin{lemma}
  \label{further} For any $x\in\Delta\setminus\partial\Delta$, \begin{align}\label{Jz1}
  \nabla H\cdot F\,(x)\,=\sum_{i,j\in\mathcal{S}}
  x_{ij}(y_{ij}-N_i(x))^2,
\end{align}
which can also be written as
\begin{align} \label{Jz2}
  \nabla H\cdot F\,(x)\,=\,\sum_{i,j\in\Sc}x_{ij}(y_{ij}-H (x))^2
-\sum_{i\in\Sc}x_i(N_i(x)-H (x))^2.\end{align}
\end{lemma}
\noindent\textbf{Remark.}\textnormal{
In the context of communication
systems, the above three formulas \eqref{eq:varh1}, \eqref{Jz1}--\eqref{Jz2} mean that  the growth rate
of the communication potential is a function depending on the difference
between efficiencies of different strategy pairs.}
\begin{proof}
Fix $i$, $j$ $\in\Sc$, and define $Y: \Sc\rightarrow \mathbb{R}$
by $Y(k):=y_{ik}$, seen as a random variable on
$(\Sc,\mathbb{P}_{z,i})$ (with expectation $\mathbb{E}_{z,i}(.)$),
where $\mathbb{P}_{z,i}(\{k\}):=x_{ik}/x_i$. Then
\begin{align*}
\mathbb{E}_{z,i}[(y_{ij}-Y)^2]&=\sum_{k\in\Sc}\frac{x_{ik}}{x_i}
\left(y_{ij}-y_{ik}\right)^2=(\mathbb{E}_{z,i}(y_{ij}-Y))^2+\mathbb{E}
_{z,i}[(Y-E_{z,i}(Y))^2]\\
&=(y_{ij}-N_i(x))^2+\sum_{k\in\Sc}\frac{x_{ik}}{x_i}(y_{ik}-N_i(x))^2,
\end{align*}
using that
$$\mathbb{E}_{z,i}(Y)=\sum_{k\in\Sc}\frac{x_{ik}}{x_i}y_{ik}=N_i(x).$$
Therefore
\begin{align*}
\nabla H \cdot
F \,(x)&=\frac{1}{2}\left(\sum_{i,j\in\Sc}x_{ij}(y_{ij}-N_i(x))^2
+\sum_{i,j,k\in\Sc}\frac{x_{ij}x_{ik}}{x_i}(y_{ik}-N_i(x))^2\right)\\
&=\sum_{i,j\in\Sc}x_{ij}(y_{ij}-N_i(x))^2,
\end{align*}
which implies that
$$\nabla H \cdot F \,(x):=\sum_{i,j\in\Sc}x_{ij}(y_{ij}-H (x))^2
-\sum_{i\in\Sc}x_i(N_i(x)-H (x))^2,$$ and completes the
proof.
\end{proof}

\begin{lemma}\label{SG:Lem:1}
  For all $x\in\Delta\setminus\partial\Delta$ and $i\in\mathcal{S}$, $N_i(x)\ge 1$.
\end{lemma}
\begin{proof}
Indeed,
$$\sum_{j\in\Sc}y_{ij}x_j=\sum_{j\in\Sc}\frac{x_{ij}}{x_i}=1$$
which subsequently implies, by Cauchy-Schwartz inequality, that
\begin{align}\label{est:N}
1=\left(\sum_{j\in\Sc}y_{ij}x_j\right)^2\le\left(\sum_{j\in\Sc}y_{ij}^2x_j\right)
\left(\sum_{j\in\Sc}x_j\right)=N_i(x).\end{align}
\end{proof}

Let us
define
$$\Delta_{\epsilon}:=\{x\in\Delta\setminus\partial\Delta:\,
p(x)>\epsilon\}\,,$$ and let
\begin{align*}
\Lambda:=\{x\in\Delta:\,p(x)=0\},
\end{align*}
where $p$ is defined in the statement of Proposition \ref{Prop:Lya}. The following Lemma \ref{restpoint} is straightforward.
\begin{lemma}\label{restpoint}
$x \in \Lambda$ if and only if $$
\frac{x_{ij}}{x_j}=\frac{x_{ik}}{x_k}, \textrm{ for all $i,j,k$ s.t.
} x_{ij}\ne 0, x_{ik} \ne 0 $$ or, equivalently,
$$y_{ij}=y_{ik}, \textrm{ for all $i,j,k$ s.t. } x_{ij}\ne 0, x_{ik}
\ne 0.
$$
\end{lemma}
\noindent\textbf{Remark.}
  \textnormal{
 Lemma \ref{restpoint} can be phrased as follows: if $x\in\La$ then, in the graph $\Gc_x$, edges within the same connected component have the same weight. Note that $x\in\G$ is equivalent to all edges of  $\Gc_x$ having the same weight $H(x)$. So the two sets $\G$ and $\La$ are different, i.e. $H$ is not a strict Lyapounov function, which justifies the need to prove separately the convergence to the set of equilibria in Section \ref{sec:eq}. }

\subsection{Proof of Theorem \ref{theorem3} and convergence to $\La$}
\label{sec:stoc}
\subsubsection{Proof of Theorem \ref{theorem3} }
For simplicity, we let $V:=V_n$ in the following calculation. Let us compute the expected increment of $(H(x_n))_{n \in
\mathbb{N}}$.
\begin{align}
&\quad\mathbb{E}[\,H(x_{n+1})-H(x_n)\,|\,\mathcal{F}_n\,]\\
&=\sum_{i\in \mathcal{S}_1, j\in
\mathcal{S}_2}\frac{V_{ij}^2}{V_iV_j}
\left(\frac{(V_{ij}+1)^2}{(V_i+1)(V_j+1)}-\frac{V_{ij}^2}{V_iV_j}\right.
\nonumber\\& \left.\qquad \qquad+\sum_{k\in\mathcal{S}, k\ne
j}\Big(\frac{V_{ik}^2}{(V_i+1)V_k}-\frac{V_{ik}^2}{V_iV_k}\Big)
+\sum_{k\in\mathcal{S}, k\ne
i}\Big(\frac{V_{kj}^2}{(V_j+1)V_k}-\frac{V_{kj}^2}{V_jV_k}\Big)\right)\nonumber\\
&=\frac{1}{2}\sum_{(i,j)\in \mathcal{S}^2}\frac{V_{ij}^2}{V_iV_j}
\left(\frac{(V_{ij}+1)^2}{(V_i+1)(V_j+1)}-\frac{V_{ij}^2}{V_iV_j}
\right. \nonumber\\& \left. \qquad \qquad+\sum_{k\in\mathcal{S},
k\ne j}\Big(\frac{V_{ik}^2}{(V_i+1)V_k}-\frac{V_{ik}^2}{V_iV_k}\Big)
+\sum_{k\in\mathcal{S}, k\ne
i}\Big(\frac{V_{jk}^2}{(V_j+1)V_k}-\frac{V_{jk}^2}{V_jV_k}\Big)\right)\nonumber\\
&=\frac{1}{2}\sum_{(i,j)\in
\mathcal{S}^2}\frac{V_{ij}^2}{V_iV_j}\left(-\sum_{k\in\mathcal{S}}\frac{V_{ik}^2}
{V_i(V_i+1)V_k}-\sum_{k\in\mathcal{S}}\frac{V_{jk}^2}{V_j(V_j+1)V_k}
\right. \nonumber\\& \left.\qquad
\qquad+\frac{(V_{ij}+1)^2}{(V_i+1)(V_j+1)}+\frac{V_{ij}^2}{V_iV_j}-
\frac{V_{ij}^2}{(V_i+1)V_j}-\frac{V_{ij}^2}{(V_j+1)V_i} \right)\nonumber\\
\label{stochlya1}
&=-\sum_{(i,j,k)\in\mathcal{S}^3}\frac{V_{ij}^2V_{ik}^2}{V_i^2(V_i+1)V_jV_k}+
\frac{1}{2}\sum_{(i,j)\in
\mathcal{S}^2}\frac{V_{ij}^2}{V_iV_j}\frac{2V_iV_jV_{ij}+V_iV_j+V_{ij}^2}{V_iV_j(V_i+1)(V_j+1)}
\nonumber\\
&=\sum_{(i,j)\in \mathcal{S}^2}\frac{V_{ij}^3}{V_i(V_i+1)V_j^2}
-\sum_{(i,j,k)\in\mathcal{S}^3}\frac{V_{ij}^2V_{ik}^2}{V_i^2(V_i+1)V_jV_k}\\
\label{stochlya2} &\qquad+ \frac{1}{2}\sum_{(i,j)\in
\mathcal{S}^2}\frac{V_{ij}^2}{V_iV_j}\frac{2V_iV_jV_{ij}+V_iV_j+V_{ij}^2}{V_iV_j(V_i+1)(V_j+1)}
-\sum_{(i,j)\in \mathcal{S}^2}\frac{V_{ij}^3}{V_i(V_i+1)V_j^2}
\end{align}

Now let us prove $(\ref{stochlya1})$ is nonnegative. Indeed,
\begin{align*}
(\ref{stochlya1})&=\sum_{(i,j,k)\in\mathcal{S}^3}\frac{V_{ij}V_{ik}}{V_i+1}\left(\frac{V_{ij}^2}{V_i^2V_j^2}-\frac{V_{ij}V_{ik}}{V_i^2V_jV_k}
\right)\\
&=\frac{1}{2}\sum_{(i,j,k)\in\mathcal{S}^3}\frac{V_{ij}V_{ik}}{V_i+1}\left(\frac{V_{ij}}{V_iV_j}-\frac{V_{ik}}{V_iV_k}\right)^2\ge0.
\end{align*}

Next we show that $(\ref{stochlya2})$ is nonnegative as well:
\begin{align*}
(\ref{stochlya2})&=-2\sum_{(i,j)\in
\mathcal{S}^2}\frac{V_{ij}^3}{V_i(V_i+1)V_j^2(V_j+1)}+\sum_{(i,j)\in
\mathcal{S}^2}\frac{V_{ij}^2}{V_iV_j(V_i+1)(V_j+1)}\\& \qquad
\qquad+\sum_{(i,j)\in
\mathcal{S}^2}\frac{V_{ij}^4}{V_i^2V_j^2(V_i+1)(V_j+1)}\\
&=\sum_{(i,j)\in
\mathcal{S}^2}\frac{V_{ij}^2}{V_iV_j(V_i+1)(V_j+1)}\left(-\frac{2V_{ij}}{V_j}+1+\frac{V_{ij}^2}{V_iV_j}\right)\\
&=\sum_{i \in \mathcal{S}_1, j\in
\mathcal{S}_2}\frac{2V_{ij}^2}{V_iV_j(V_i+1)(V_j+1)}\left(1-\frac{V_{ij}}{V_i}
-\frac{V_{ij}}{V_j}+\frac{V_{ij}^2}{V_iV_j}\right)\\
&=\sum_{i \in \mathcal{S}_1, j\in
\mathcal{S}_2}\frac{2V_{ij}^2}{V_iV_j(V_i+1)(V_j+1)}\Big(1-\frac{V_{ij}}{V_i}\Big)\Big(1-\frac{V_{ij}}{V_j}\Big)\ge0.
\end{align*}
\subsubsection{Convergence to $\La$}
\label{SEC:4:3}
Let us now prove the following Proposition \ref{StochLya}.
\begin{proposition}
\label{StochLya} $(x_n)_{n\in\mathbb{N}}$ converges to
$\Lambda$ a.s. ; more precisely, $(p(x_n))_{n\in\mathbb{N}}$ converges to $0$ a.s.
\end{proposition}
\begin{proof} Let us define process $Y_n:=H(x_n), \, n\in\mathbb{N}$. We
decompose $(Y_n)_{n \in \mathbb{N}}$ into a martingale
$(M_n)_{n\in\mathbb{N}}$ and a predictable process
$(A_n)_{n\in\mathbb{N}}$ where
$A_{n+1}-A_n=\mathbb{E}[\,Y_{n+1}-Y_n\,|\,\mathcal{F}_n\,].$ Since
$H$ is bounded, martingale $(M_n)_{n \in \mathbb{N}}$ is upper
bounded and hence converges.

Let
\begin{align*}
P_n&:=
\frac{1}{2}\sum_{(i,j,k)\in\mathcal{S}^3}\frac{V_{ij}^nV_{ik}^n}{V_i^n+1}\left(\frac{V_{ij}^n}{V_i^nV_j^n}-\frac{V_{ik}^n}{V_i^nV_k^n}\right)^2
;\\
Q_n&:=
\sum_{i \in \mathcal{S}_1, j\in
\mathcal{S}_2}\frac{(V_{ij}^n)^2}{V_i^nV_j^n(V_i^n+1)(V_j^n+1)}\Big(1-\frac{V_{ij}^n}{V_i^n}\Big)\Big(1-\frac{V_{ij}^n}{V_j^n}\Big).
\end{align*}
Hence
\begin{align*}
P_n\ge
\frac{1}{4}\sum_{(i,j,k)\in\mathcal{S}^3}\frac{V_{ij}^nV_{ik}^n}{V_i^n}
\left(\frac{V_{ij}^n}{V_i^nV_j^n}-\frac{V_{ik}^n}{V_i^nV_k^n}\right)^2
\,=\,\frac{p(x_n)}{2\,T_n}.
\end{align*}

The rest of the argument is similar to the proof of convergence to the set of equilibria in [1]. If $(x_n)_{n\in\mathbb{N}}$ were infinitely often away from
$\Lambda$, then the drift would cause $(H(x_n))_{n\in\mathbb{N}}$ to
go to infinity, hence contradicting the boundedness of $H$. Indeed,
let $\delta$ be the distance between $\Delta_{\epsilon}$ and the
complement of $\Delta_{\epsilon/2}$. Suppose
$x_n\in\Delta_\epsilon$, $x_{n+1},\ldots,x_{n+k-1}\in
\Delta_{\frac{\epsilon}{2}}\setminus\Delta_{\epsilon}$ and
$x_{n+k}\in\Delta_{\frac{\epsilon}{2}}^c$,
\begin{align*}
A_{n+k}-A_n\,&\ge\,\sum_{r=n}^{n+k-1}(P_r+Q_r)
\ge\,\sum_{r=n}^{n+k-1}P_r\,=\,\sum_{r=n}^{n+k-1}\frac{p(x_r)}{2T_r}
\ge\,\frac{\epsilon}{4}\sum_{r=n}^{n+k-1}\frac{1}{T_r+1}.
\end{align*}
Therefore,
\begin{align*}
\delta\,&\le\,\sum_{r=1}^{n+k-1}|x_{r+1}-x_r|
\le \, \sum_{r=1}^{n+k-1}\frac{1}{1+T_r}
\le \, \frac{4}{\epsilon}(A_{n+k}-A_n).
\end{align*}
Therefore $x_n\in \Delta_\epsilon$ infinitely often would cause $A_n$
to increase to infinity. By contradiction, $(x_n)_{n\in \mathbb{N}}$
must converge to $\Lambda$. \end{proof}
\noindent\textbf{Remark.}\textnormal{
The proof of convergence of $p(x_n)$ to $0$ would also hold on (deterministic) solutions of the ODE.
}

\begin{corollary} \label{SG:COR:1} Almost surely,
  $$\frac{T_n}{n} \rightarrow \lim_{n\to\iy}\frac{H(x_n)}{M_1} \,\in\,\left[\frac{1}{M_1},\frac{\min(M_1,M_2)}{M_1}\right]\quad
  \textnormal{\textnormal{as}} \,\, n\to \infty.$$
\end{corollary}
\bp
Same as Lemma \ref{basic:est}.
\ep

\section{Equilibria}
\label{SEC:4:5}
\subsection{Convergence to the set of equilibria : Proof of Theorem \ref{theorem5}}
\label{sec:eq}
We already know that the occupation measure
$(x_n)_{n\in\mathbb{N}}$ a.s. converges to $\Lambda$; the goal of this section is to
prove that, more precisely, $(x_n)_{n\in\mathbb{N}}$ a.s.  converges
to the set of equilibria $\Gamma$ of the ODE.
\begin{lemma}
\label{est1} Suppose $\epsilon$ small enough and
$x\in\Delta_{\epsilon^4}^c$. Then for any $i\in\mathcal{S}_1,\,
j\in\mathcal{S}_2$ s.t. $x_{ij}>\epsilon$,
\begin{align*}
\left|\,y_{ij}-N_i(x)\,\right|<\frac{\epsilon}{3}\,, \qquad \textrm{
and }\qquad \left|\,y_{ij}-N_j(x)\,\right|<\frac{\epsilon}{3}\,.
\end{align*}
\end{lemma}
\begin{proof}
Follows directly from (\ref{Jz1}).
\end{proof}

\begin{lemma}
\label{stochapp2}
\begin{align}
\label{sterm1} y_{ij}^{n+1}-y_{ij}^{n}
=\frac{1}{M_1T_n}\,y_{ij}^n\,\Big(y_{ij}^n-N_i(x_n)-N_j(x_n)+H(x_n)\Big)
+r_{n+1}+\zeta_{ij}^{n+1},
\end{align}
where $(r_n)_{n\ge1}$ is predictable, $\mathbb{E}\,[\,\zeta_{ij}^{n+1}\,|\,\mathcal{F}_n\,]=0$ and
\begin{align}\label{est5}
|\zeta_{ij}^{n+1}|<\frac{12}{T_n
  x_i^nx_j^n},
\quad |r_{n+1}|<\frac{6(x_i^n+x_j^n)}{(T_nx_i^nx_j^n)^2}\le\frac{12}{(T_nx_i^nx_j^n)^2}.\end{align}
\end{lemma}
\begin{proof}

\begin{align}
 y_{ij}^{n+1}-y_{ij}^{n}
&=\frac{T_{n+1}V_{ij}^{n+1}}{V_i^{n+1}V_j^{n+1}}-\frac{T_nV_{ij}^n}{V_i^nV_j^n}
\nonumber\\
\label{sterm2}
&=(T_{n+1}V_{ij}^{n+1}V_i^nV_j^n-T_nV_{ij}^nV_i^{n+1}V_j^{n+1})\\
\label{sterm3} &\quad \times
\frac{1}{(V_i^nV_j^n)^2}\left(1+\Big(\frac{V_i^nV_j^n}{V_i^{n+1}V_j^{n+1}}-1\Big)\right)
\end{align}
and
\begin{align*}
  (\ref{sterm2})\,&=\,T_nV_i^nV_j^n\1_{\Delta
V_{ij}^{n+1}>0}\,+\,V_{ij}^nV_i^nV_j^n\1_{\Delta
T_{n+1}>0}\,+\,V_i^nV_j^n\1_{\Delta V_{ij}^{n+1}>0}\\
&\qquad \,-\,T_nV_{ij}^nV_i^n\mathbf{1}y_{\Delta
V_i^{n+1}>0}\,-\,T_nV_{ij}^nV_j^n\mathbf{1}y_{\Delta
V_j^{n+1}>0}\,-\,T_nV_{ij}^n\1_{\Delta V_{ij}^{n+1}>0}.
\end{align*}
Hence
\begin{align*}
  |(\ref{sterm2})|&\,\le\,6\,T_nV_i^nV_j^n,
  \end{align*}
and
\begin{align*}
  \mathbb{E}\left[\,\frac{(\ref{sterm2})}{(V_i^nV_j^n)^2}\,\Big|\,\mathcal{F}_n\,\right]
  \,=\,\frac{1}{M_1T_n}\,y_{ij}^n\,\big(y_{ij}^n-N_i(x_n)-N_j(x_n)+H(x_n)\big)\,.
\end{align*}
By the following simple estimate $$
  \frac{V_{i}^{n}V_{j}^n}{V_i^{n+1}V_j^{n+1}}-1\,\le\,\frac{1}{V_i^{n+1}}+\frac{1}{V_j^{n+1}},$$
we deduce that
\begin{align*}
  |r_{n+1}|\,=\,\left|\mathbb{E}\left[\frac{(\ref{sterm2})}{(V_i^nV_j^n)^2}
  \left(\frac{V_i^nV_j^n}{V_i^{n+1}V_j^{n+1}}-1\right)\,\Big|\,\mathcal{F}_n\right]\right|\\
\,\le\,
\frac{6\,T_n(V_i^n+V_j^n)}{(V_i^nV_j^n)^2}\,\le\,\frac{6(x_i^n+x_j^n)}{(T_nx_i^nx_j^n)^2}\,.
\end{align*}
and
\begin{align}
\label{est3}
  |y_{ij}^{n+1}-y_{ij}^n|\,=\,(\ref{sterm2})\times (\ref{sterm3})\,\le\, \frac{6}{T_n
  x_i^nx_j^n}.
\end{align}
Therefore
$$|\zeta_{ij}^{n+1}|\,=\,2\,|y_{ij}^{n+1}-y_{ij}^n-\Es[y_{ij}^{n+1}-y_{ij}^n|\F_n]|\le\, \frac{12}{T_n
  x_i^nx_j^n}.$$
\end{proof}

Let
\begin{align*}U_{ij}(\epsilon)&:=\Big\{\,x\in\Delta\,:\,x_{ij}\le\epsilon\textrm{
or } y_{ij}-H(X)\ge -\epsilon\,\Big\}.\\
  \end{align*}

\begin{lemma}
\label{lemma:supermart}
  Assume $\epsilon>0$ is small enough and $m_0\in\mathbb{N}$ is
  large enough. Let
  \begin{align*}
    &R_n:=\sum_{m=m_0}^n\Big(y_{ij}^m-y_{ij}^{m-1}-\frac{\epsilon^2}{6m}\Big)\1_{x_{m-1}\notin
  U_{ij}(\epsilon)\cup\Delta_{\epsilon^{4}},\,T_{m-1}\ge
  \frac{m}{2M_1}},\quad \forall\,n\in\mathbb{N};\\
  &S_n:=\sum_{m=m_0}^n\Big(x_{ij}^m-x_{ij}^{m-1}+\frac{\epsilon}{m}\Big)\1_{x_{m-1}\notin
  U_{ij}(\epsilon)\cup\Delta_{\epsilon^{4}},\,T_{m-1}\ge
  \frac{m}{2M_1}},\quad \forall\,n\in\mathbb{N}.
  \end{align*} Then
  \begin{description}
  \item[(a)] $(R_n)_{n\in\mathbb{N}}$ (resp. $(S_n)_{n\in\mathbb{N}}$) is a submartingale (resp.
  supermartingale).
  \item[(b)] $\limsup_{n\ge m, m\to\iy}(R_n-R_m)^-=\limsup_{n\ge m, m\to\iy}(S_n-S_m)^+=0$.
  \end{description}
\end{lemma}

\begin{proof}
  First we note that if $x_n\notin U_{ij}(\epsilon)$,
  $$\mathbb{E}[\,x_{ij}^{n+1}-x_{ij}^n\,|\,\mathcal{F}_n\,]=\frac{x_{ij}^n}{1+T_n}(y_{ij}^n-H(x_n))\le -\frac{\epsilon}{T_n}.$$
  This implies that $(S_n)_{n\in\mathbb{N}}$ is a supermartingale. Now we prove that
  $(R_n)_{n\in\mathbb{N}}$ is a submartingale.

Assume $\epsilon>0$ small enough and $x_n\notin U_{ij}(\epsilon)\cup
\Delta_{\epsilon^{4}}$. Then Lemma \ref{est1} implies that
$$y_{ij}^n-N_i(x_n)-N_j(x_n)+H(x_n)\,\ge\,\frac{\epsilon}{3}.$$
Hence,
\begin{align*}
  \mathbb{E}[\,y_{ij}^{n+1}-y_{ij}^n\,|\,\mathcal{F}_n\,]\,&\ge\,
  \frac{1}{M_1T_nx_i^nx_j^n}\left(\frac{x_{ij}^n\epsilon}{3}-\frac{12}{T_nx_i^nx_j^n}\right)\\
  &\ge\,\frac{x_{ij}^n\epsilon}{6M_1T_nx_i^nx_j^n}\,\ge\,\frac{\epsilon^2}{6n},
\end{align*}
if $n\ge 144M_1/\epsilon^4$ (which implies $T_n\ge 72/\epsilon^4$ and
therefore $T_nx_i^nx_j^n\ge 72/(x_{ij}^n\epsilon)$ ).

Let us now prove $(b)$. Let
\begin{align*}
  \Pi_n&:=R_n-\sum_{m=m_0}^n\mathbb{E}[\,R_m-R_{m-1}\,|\,\mathcal{F}_{m-1}\,],\\
  \Xi_n&:=S_n-\sum_{m=m_0}^n\mathbb{E}[\,S_m-S_{m-1}\,|\,\mathcal{F}_{m-1}\,].
\end{align*}

By (\ref{est5}), we note that for all $n\ge m_0$,
\begin{align*}
  \mathbb{E}[\,(\Pi_{n+1}-\Pi_n)^2\,|\,\mathcal{F}_n\,]&\,\le\,
  \mathbb{E}[\,(y_{ij}^{n+1}-y_{ij}^n)^2\,|\,\mathcal{F}_n\,]\\
  &\,\le\,
  \frac{12}{(T_nx_i^nx_j^n)^2}\1_{x_n\notin
  U_{ij}(\epsilon), T_n\ge \frac{n}{2M_1}}\,\le\,
  \frac{48M_1^{\,2}}{\epsilon^4 n^2}.
\end{align*}
Therefore $(\Pi_n)_{n\in\mathbb{N}}$ is bounded in
\emph{$\mathbb{L}^2$} and hence converges. We can obtain similar
bounds for $(\Xi_n)_{n\in\mathbb{N}}$, which converges as
well. This completes the proof.
\end{proof}

\begin{lemma}
\label{2epsilon}
  Let $\epsilon>0$, and assume $n\in\mathbb{N}$ is sufficiently
large (depending on $\epsilon$). If $x_n\in U_{ij}(\epsilon)$,
  $|H(x_{n+1})-H(x_{n})|<\epsilon/2$ and $T_n\ge
  n/(2M_1)$, then $x_{n+1}\in U_{ij}(2\epsilon).$
\end{lemma}
\begin{proof}
  Let $x_n\in U_{ij}(\epsilon)$.
  Then $(\ref{s-est2})$ implies
  $$|x_{ij}^{n+1}-x_{ij}^n|\le\frac{2M_1}{n}.$$

 Assume that $n$ is large enough. If $x_{ij}^n\le\e$, then $x_{ij}^{n+1}\le2\epsilon$. Otherwise $x_{ij}^n>\e$ and $y_{ij}^n-H(x_n)\ge-\e$; assuming $T_n\ge
  n/(2M_1)$ and using (\ref{est3}), we have
  $$|y_{ij}^{n+1}-y_{ij}^n|\le \frac{6}{T_nx_i^nx_j^n}\le
  \frac{12M_1}{\epsilon^2 n } < \frac{\epsilon}{2}$$
and, by assumption $|H(x_{n+1})-H(x_{n})|<\epsilon/2$, so that we conclude that $y_{ij}^{n+1}-H(x_{n+1})\ge-2\e$.
\end{proof}

\begin{lemma}
\label{lemma:setofequ}
  $$\limsup_{n\to \infty} x_{ij}^n(y_{ij}^n-H(x_n))^-=0.$$
\end{lemma}
\begin{proof}
  We fix $\epsilon>0$ and $m_0\in\mathbb{N}$, and let $\tau_{m_0}$
  be the stopping time
  $$\tau_{m_0}:=\inf\Big\{n\ge m_0\,:\, x_n\notin
  \Delta_{\epsilon^{4}} \textrm{ or } T_n<\frac{n}{2M_1} \textrm{ or
  } |H(x_n)-H(x_{m_0})|>\frac{\epsilon}{4}\Big\}.$$
  We only need to show that almost surely, either $\tau_{m_0}<\infty$ or $x_n\in
  U_{ij}(3\epsilon)$ for all $n$ large enough. This will complete the proof since we know from Theorem \ref{theorem3}, Lemma \ref{basic:est}
  and Proposition \ref{StochLya} that there exists almost surely
  $m_0\in\mathbb{N}$ s.t. $\tau_{m_0}=\infty.$

Let $\sigma_{m_0}$ be the stopping time
$$\sigma_{m_0}:=\inf\{n\ge m_0 \,:\, x_n\in U_{ij}(\epsilon)\}.$$
Lemma $\ref{lemma:supermart}(b)$ implies that there exists a.s. a
(random) $m_0\in\mathbb{N}$ such that, for all $n\ge m \ge m_0$,
\begin{align} \label{RSest} (R_n-R_m)^-\le \frac{\epsilon}{2},\qquad(S_n-S_m)^+\le
\frac{\epsilon}{2}.\end{align} Therefore  $\tau_{m_0}<\infty$
or $\sigma_{m_0}<\infty$, using $\sum_{n\ge m_0}1/n=\infty$ and the
observation that $x_{ij}^n$ is bounded.

For all $n\in[\sigma_{m_0},\tau_{m_0})$, let $\rho_n$
be the largest $k\le n$ such that $x_k\in U_{ij}(\epsilon)$. By
(\ref{RSest}), $$y_{ij}^n-y_{ij}^{\rho_n+1}\ge
-\frac{\epsilon}{2}.$$
Now $x_{\rho_n+1}\in
U_{ij}(2\epsilon)$ (by Lemma \ref{2epsilon}): let us assume for instance that $y_{ij}^{\rho_n+1}-H(x_{\rho_n+1})\ge-2\e$. Together with
$|H(x_n)-H(x_{\rho_n+1})|\le \epsilon/2$ (by $n<\tau_{m_0}$), we
deduce that
\begin{align*}
  y_{ij}^n-H(x_n)\ge y_{ij}^{\rho_n+1}-H(x_{\rho_n+1})-\epsilon\ge -3\epsilon.
\end{align*}
With a similar argument, $x_{ij}^{\rho_n+1}\le 2\epsilon$ implies
$x_{ij}^n\le 3\epsilon$. So overall, $x_n\in U_{ij}(3\epsilon)$ if
$\sigma_{m_0}\le n<\tau_{m_0}$, which enables us to conclude.

\end{proof}

\noindent \textbf{Proof of Theorem \ref{theorem5}.}
\begin{align*}
2H(x_n)\,&=\,\sum_{i,j\in\Sc}y_{ij}^nx_{ij}^n\\
&=\, 2H(x_n)+\sum_{i,j\in\Sc}(y_{ij}^n-H(x_n))^+x_{ij}^n-\sum_{i,j\in\Sc}(y_{ij}^n-H(x_n))^-x_{ij}^n
\end{align*}
Hence, $\lim_{n\to \infty} x_{ij}^n(y_{ij}^n-H(x_n))^{-}=0$
implies$$\lim_{n\to \infty} x_{ij}^n(y_{ij}^n-H(x_n))=0.$$ Lemma
\ref{lemma:setofequ} enables us to conclude.

\subsection{Bipartite graph structure}
\label{SEC:4:6}
Let us recall the bipartite graph defined in Section \ref{SEC:4:1} (see
Definition \ref{graph.1}): any $x\in \Delta$ is associated
with a weighted bipartite graph $\mathcal{G}_x$ with vertices
$\mathcal{S}:=\mathcal{S}_1\cup\mathcal{S}_2$, adjacency $\sim$ and
weights as follows \begin{description}
\item[(1)] $\forall \,i \in \mathcal{S}_1, j \in
\mathcal{S}_2, i\sim j$ if and only if $x_{ij}>0.$ \item[(2)] The
weight of edge $\{i,j\}$ is $x_{ij}/(x_ix_j)$.
\end{description}

Let $\Cc_1$, $\ldots$, $\Cc_d$ be the connected components of
$\Gc_x$. Besides the bipartite graph defined above, let us also discuss two
other possible ways to assign weights:
\begin{description}
\item[Graph $\mathcal{G}_x^1$]:$\quad$ Edge $e_{ij}$ has weight $x_{ij}/x_j$, $i \in \mathcal{S}_1, j\in \mathcal{S}_2$.
\item[Graph $\mathcal{G}_x^2$]:$\quad$ Edge $e_{ij}$ has weight $x_{ij}/x_i$, $i \in \mathcal{S}_1, j\in \mathcal{S}_2$.
\end{description}
By Lemma \ref{restpoint}, we observe some interesting properties of these three
graphs when $x\in \Lambda$, in particular $x\in \Gamma$:
\begin{description}
\item[On $\mathcal{G}_x$]:$\,$ All edges in a component
$\mathcal{C}_k,\, k=1,\ldots,d$ have the same weight $\lambda_k$. Hence, $$H(x)=\sum_{k=1}^d\sum_{i \in
\mathcal{S}_1\cap \mathcal{C}_k}x_i\lambda_k=\sum_{k=1}^d\sum_{j \in
\mathcal{S}_2\cap \mathcal{C}_k}x_j\lambda_k.$$ Furthermore, if $x$
is an equilibrium, all the edges in $\mathcal{G}_x$ have the same
weight, which equals $H(x)$.
\item[On $\mathcal{G}_x^1$]:$\,$ Every edge linked with the same state $i$ has the same weight, which we denote by
$k_i$. Also for each signal $j$, the sum of weights of edges linked
to $j$ is equal to $1$, so that $H(x)=\sum_{i=1}^{M_2}k_i$.
\item[On $\mathcal{G}_x^2$]:$\,$ Every edge linked with the same signal $j$ has the same weight, which we denote by $k'_j$. Also for each state $i$, the sum of the weights of edges linked
to $i$ is equal to $1$. $H(x)=\sum_{j=1}^{M_1}k'_j$.
\end{description}

\subsection{Properties of Lyapounov function}
\label{SEC:4:7}

We now show that $H$ is constant on each
connected component of $\Gamma$. Since it is not continuous on the boundary, we first prove in Lemma
\ref{samesupp} that it takes a constant value on connected subsets
of $\Gamma$ with the same support (defined below) by a
differentiability argument, and then conclude in Proposition
\ref{constantH} by a continuity argument on the set of equilibria.

Let
\begin{align*}
    \Theta:=\{\theta\,:\,\theta \,\subseteq\, \mathcal{S}_{\textnormal{pair}}\}.  \end{align*}
For any $x\in\Delta$, we define its support
$$S_x:=\{(i,j)\,:\, i \in \mathcal{S}_1, \,j \in \mathcal{S}_2,
 \, x_{ij}>0\}.$$ $\Theta$ can be used as an index set to divide
$\Delta$ into several parts in the following sense: for any $\theta
\in \Theta,$
\begin{align*}
  \Delta_{\theta}&:=\{x\in\Delta:\,S_x=\theta\},\\
  \Gamma_{\theta}&:=\Delta_{\theta}\,\cap\,\Gamma.
\end{align*}
\begin{lemma}
\label{samesupp}
  For any $\theta\,\in\,\Theta$, $H$ is constant on each component of
  $\Gamma_{\theta}$.
\end{lemma}
\begin{proof}
Given $q\in\G_\Theta$, let us differentiate $H$ at $q$ with respect to $x_{ij}=x_{ji},\,(i,j)\in
S_q$ without the constraint $x\in\Delta$:
\begin{align*} &\left(\frac{\partial }{\partial
x_{ij}}H(x)\right)_{x=q}=\sum_{(k,l)\in S_q}\frac{\partial
}{\partial
x_{ij}}\left(\frac{ x_{kl}^2}{x_kx_l}\right)\\
&=\frac{\partial}{\partial x_{ij}}\left(\frac{
x_{ij}^2}{x_ix_j}\right) +\sum_{k\ne j,\,(i,k)\in
S_q}\frac{\partial}{\partial x_{ij}}\left(\frac{
x_{ik}^2}{x_ix_k}\right) +\sum_{l\ne i,\,(l,j)\in
S_q}\frac{\partial}{\partial
x_{ij}}\left(\frac{ x_{lj}^2}{x_lx_j}\right)\\
&=\frac{q_{ij}}{q_iq_j}\left(2-\frac{q_{ij}}{q_i}-\frac{q_{ij}}{q_j}\right)
+\sum_{k\ne j,\,(i,k)\in
S_q}\frac{q_{ik}}{q_iq_k}\left(-\frac{q_{ik}}{q_i}\right)
+\sum_{l\ne i,\,(l,j)\in S_q}\frac{q_{lj}}{q_lq_j}\left(-\frac{q_{lj}}{q_j}\right)\\
&= H(q)\left(2-\frac{q_{ij}}{q_i}-\frac{q_{ij}}{q_j}\right)+
H(q)\left(-\frac{q_i-q_{ij}}{q_i}\right) + H(q)\left(-\frac{q_j-q_{ij}}{q_j}\right)\\&=0 .
\end{align*}
The penultimate equality comes from the fact that  $
q_{ij}/(q_iq_j)= H(q)$ if $(i,j)\in S_q,\,q\in \Gamma$.
\end{proof}

\begin{proposition}
\label{constantH} $H$ is constant on each connected
component of $\Gamma$.
\end{proposition}
\begin{proof} Let us show that $H$ is continuous on $\Gamma$, which will enable us to conclude. Indeed, suppose that $q\in \Gamma$,
and  that $x\in \Gamma$ is in the neighbourhood of $q\in \Gamma$ within
$\Delta$, then $S_x \supseteq S_q$ and, using $x\in \Gamma$,
\begin{align*}
 H(x)=\sum_{(i,j)\in S_q} \frac{x_{ij}^2}{x_ix_j}+\sum_{(i,j)\in S_x\setminus S_q}x_{ij}H(x),
\end{align*}
so that
$$H(x)=\frac{1}{1-\sum_{(i,j)\in S_x\setminus S_q} x_{ij}}\sum_{(i,j)\in S_q} \frac{x_{ij}^2}{x_ix_j},$$
and the conclusion follows. \end{proof}

\subsection{Classification of equilibria and stability}
\label{SEC:4:8}
\subsubsection{Jacobian matrix} \label{SEC:4:9}
At any equilibrium $x\in(\Delta\setminus\partial\Delta)\cap\G$ ($F$ is not differentiable on $\partial\De$), we
calculate the Jacobian matrix
$$J_x=\left(\frac{\partial\,F_{lk}}{\partial
x_{ij}}\right)_{(i,j), (l,k)\in \Sc_{\textnormal{pair}}}$$ where, by a slight abuse of notation,
$$F(x)=(F_{ij}(x))_{(i,j)\in\mathcal{S}_{\textnormal{pair}}}.$$

For all $(i,j)$, $(l,k)$ $\in\Sc_{\textnormal{pair}}$, a simple extension of the calculation in the proof of Lemma \ref{samesupp} yields
$\frac{\partial\, H}{\partial x_{ij}}(x)=-\1_{\{x_{ij}=0\}}2H(x)$, so that
\begin{align*}
&\frac{\partial\,F_{lk}}{\partial
x_{ij}}=\1_{\{(i,j)=(l,k), x_{ij}=0\}}(y_{lk}-H(x))+x_{lk}\frac{\partial\,y_{lk}}{\partial
x_{ij}}-x_{lk}\frac{\partial\, H}{\partial
x_{ij}}(x)\\
&=-\1_{\{(i,j)=(l,k), x_{ij}=0\}}H(x)+x_{lk}\frac{\partial\,y_{lk}}{\partial
x_{ij}}+2x_{lk}\1_{\{x_{ij}=0\}}H(x)\\
&=H(x)\1_{\{(i,j)=(l,k)\}}(\1_{\{x_{ij}\ne0\}}-\1_{\{x_{ij}=0\}})\\ &+H(x)\1_{\{x_{lk}\ne0\}}
\left(-\frac{x_{ik}}{x_i}\1_{\{i=l\}}-\frac{x_{lj}}{x_j}\1_{\{j=k\}}
\right)+2x_{lk}\1_{\{x_{ij}=0\}}H(x)
\end{align*}

Therefore, for any $(i,j)\in\mathcal{S}_{\textnormal{pair}}$ s.t. $x_{ij}\ne0$,
\begin{align}\label{matrix1}
\frac{\partial F_{ij}}{\partial x_{ij}}&\,=
H(x)\left(1-\frac{x_{ij}}{x_i}-\frac{x_{ij}}{x_j}\right);\\
\frac{\partial F_{ik}}{\partial x_{ij}}&\,=
\,H(x)\left(-\frac{x_{ik}}{x_i}\right), \qquad\, k\in
\mathcal{S}_2\setminus\{j\};\\
\frac{\partial F_{lj}}{\partial x_{ij}}&\,=\,
H(x)\left(-\frac{x_{lj}}{x_j}\right), \qquad\, l\in
\mathcal{S}_1\setminus\{i\};\\
\frac{\partial F_{lk}}{\partial x_{ij}}&\,=\,0, \qquad\, l\in
\mathcal{S}_1\setminus\{i\},\,k \in \mathcal{S}_2\setminus\{j\};
\end{align}
for any $(i,j)\in\mathcal{S}_{\textnormal{pair}}$ s.t. $x_{ij}=0$,
\begin{align}
\frac{\partial F_{ij}}{\partial x_{ij}}&=-H(x); \\
\label{matrix2}\frac{\partial F_{lk}}{\partial x_{ij}}&=0, \qquad
l\in \mathcal{S}_1,k \in
\mathcal{S}_2, (l,k)\ne(i,j), x_{lk}=0.
\end{align}

 Let $\mathcal{C}_1, \ldots,\mathcal{C}_d$ be the connected components of  the edges of $\mathcal{G}_x$.
 Let
 $$J_x^m:=\left(\frac{\partial\,F_{lk}}{\partial
x_{ij}}\right)_{(i,j), (l,k)\in \mathcal{C}_m}.$$
Therefore,    using
(\ref{matrix1})-(\ref{matrix2}),
$J_x$ can be written as follows, by putting first $(i,j)$ and $(l,k)$ coordinates such that $x_{ij}\ne0$ and $x_{lk}\ne0$ (in the same order, with increasing connected components $\mathcal{C}_1, \ldots,\mathcal{C}_d$)
\begin{displaymath}
\left(
\begin{array}{cccccc}
\textbf{$J_x^1$} &  &  & && (0)\\
& \ddots &  && & \\
& & \textbf{$J_x^{d}$}& &&\\
 &  &   & -H(x) &&\\
&&&&\ddots &\\
(*)&&&&& -H(x)
\end{array}
 \right)
\end{displaymath}
\subsubsection{Classification of equilibria based on stability}
\label{SEC:4:10}
Let us introduce a few definitions of stability for
ordinary differential equations.
\begin{definition}
$x$ is \textnormal{Lyapounov stable} if for any
neighborhood \textnormal{$U_1$} of $x$, there exists a neighborhood
\textnormal{$U_2 \subseteq U_1$} of $x$ such that any solution
\textnormal{$x(t)$} starting in \textnormal{$U_2$ } is such that \textnormal{$x(t)$} remains in \textnormal{$U_2$} for all
\textnormal{$t\ge 0$}.
\end{definition}

\begin{definition}
$x$ is \textnormal{asymptotically stable} if it is \textnormal{Lyapounov stable} and there
exists a neighbourhood $U_1$ such that any solution $x(t)$
starting in $U_1$ is such that $x(t)$ converges to $x$.
\end{definition}

An equilibrium that is Lyapunov stable but not asymptotically stable
is called $neutrally$ $stable$  sometimes.

\begin{definition}
 $x$ is \textnormal{linearly stable}  if all eigenvalues of the Jacobian matrix
 at $x$ have nonpositive real part; otherwise, $x$ is called \textnormal{linearly unstable}.
\end{definition}

Remark that, with these definitions, linear stability allows for eigenvalues to have zero real part, and therefore does not necessarily imply Lyapounov stability. However the dynamics considered here makes these stable equilibria indeed Lyapounov stable when they do not lie outside the boundary of $\De$ are studying here, as it can be shown by the help of an entropy function; Section \ref{sec:>0} on convergence with positive probability to stable configurations can be understood as a consequence of this propoerty in the nondeterministic case.

\begin{definition}
Let
\begin{align*}
  \Gamma_0&\,:=\,\Gamma\,\cap\,\Delta\setminus\partial\Delta,\\
  \Gamma_b&\,:=\,\Gamma\,\cap\,\partial\Delta,
\end{align*}
and let $\G_s$ (resp. $\Gamma_u$) be the set of linearly stable (resp. unstable) equilibria in
$\Gamma_0$ for the mean-field ODE.
\end{definition}

For any $x \in \Gamma_u$, let
$$\mathcal{E}_x:=\{\theta\in\mathbb{R}^{M_1\times
M_2}\,:\,|\theta|=1 \textnormal{ and } \exists \,(i,j)\in S_x
\textnormal{ s.t. }\theta\cdot \mathbf{e}_{ij}>0\}.$$

\begin{proposition}
\label{Prop:unstable}
We have
\begin{description}
\item [(a)] $\G_s=\{x\in\G_0 : {(P)_{\Gc_x}}\text{ holds}\}$.\\
\item [(b)] If $x\in\G_u$, then there exists an eigenvector in $\mathcal{E}_x$ whose
eigenvalue has positive real part.\\
\item [(c)] For all $x\in\G_u$ and $\theta\in\mathcal{E}_x\setminus\{0\}$, there exists a neighbourhood $\Nc(x)$ of $x$ such that, if $x_n\in\Nc(x)$, then
$$\mathbb{E}[\,(\eta_{n+1} \cdot \theta)^2\,|\,\mathcal{F}_n\,]
>\frac{\mathsf{Cst}(x)}{n^2},$$
where $\eta_{n+1}$ is the martingale increment defined in
(\ref{Stochapprox}).
\end{description}
\end{proposition}
Note that {\bf(b)-(c)} will be used to that  $(x_n)_{n\in\N}$ stochastically "perturbs enough''  the ODE \eqref{ODE}, in order to prove nonconvergence to unstable equilibria.

To prove Proposition \ref{Prop:unstable}, we need following  Lemma \ref{twoedges} on the structure of $\mathcal{G}_x$ when $x\in\G_0$, and the elementary Lemma \ref{basest}.
\begin{lemma}
\label{twoedges} For any $x\,\in\, \Gamma_0$ such that
$\mathcal{P}_{\mathcal{G}_x}$ does not hold, $\mathcal{G}_x$ has at
least one connected component on which every vertex has at least two
edges.
\end{lemma}
\begin{proof}
First, $x\in \Gamma_u$ implies that there exists a connected
component $\mathcal{C}$ with at least two states and two signals.
Assume, by contradiction, that signal $j$ is in $\mathcal{C}$ and
$j$ is only linked to one state, for instance state $i$. Then
$x_{ij}/x_{j}=1$, and for all $k\in\mathcal{S}_2,\, s.t. \,
x_{ik}>0$,
$$\frac{x_{ik}}{x_k}=\frac{x_{ij}}{x_j}=1,$$ i.e. $k$ is only linked
with one edge. It implies $\mathcal{C}\cap\mathcal{S}_1=\{i\},$
which contradicts our assumption.
\end{proof}
\begin{lemma}
\label{basest}
  If a random variable $\omega$ satisfies that
  $\mathbb{E}[\omega]<\infty$,
  $\mathbb{P}(\omega=a)\ge p$ and $\mathbb{P}(\omega=b)\ge p$, then
  $$\textnormal{Var}(\omega)\ge \frac{(b-a)^2p}{4}\,.$$
\end{lemma}
\begin{proof}
  $$\textrm{Var}(\omega)\ge
  p\big((a-\mathbb{E}[\omega])^2\lor(b-\mathbb{E}[\omega])^2\big)\ge
  \frac{(b-a)^2p}{4}\,.$$
\end{proof}

\noindent\textbf{Proof of Proposition \ref{Prop:unstable}.} \textbf{(a)-(b)}
Suppose  that $x \in \Gamma_0$ and that $\mathcal{P}_{\mathcal{G}_x}$ does not
hold, and let us prove that $x\in\G_u$. Lemma
\ref{twoedges} implies that $\mathcal{G}_x$ has a connected
component on which every vertex has at least two edges, which we assume to be $\mathcal{C}_1$ w.l.o.g. Let $V(\mathcal{C}_1)$
(resp. $E(\mathcal{C}_1)$) be its set of vertices (resp. edges). Let us show that $J_x^1$ has at least one eigenvalue with positive real part.

Indeed, let us compute the trace of $J_x^1$:
\begin{align*}
Tr(J_x^1)&=H(x)\sum_{i,j\in
\mathcal{C}_1 : i\sim j }(1-\frac{x_{ij}}{x_i}-\frac{x_{ij}}{x_j})\\
&=H(x)\left(|E(\mathcal{C}_1)|-|V(\mathcal{C}_1)|\right)\,\ge\, 0.
\end{align*}
The last inequality comes from the fact that the number of edges is
greater than or equal to the number of vertices in $\mathcal{C}_1$
because every vertex has at least two edges in $\mathcal{C}_1$.

Now it is easy to check that $(J_x^1)
\mathbf{1}=-H(x)\1$ where $\mathbf{1}=(1,\ldots,1)^T$, and therefore that $-H(x)$ is an eigenvalue of $J_x^1$, which enables us to conclude (b) and the first part of (a).

 Now suppose that $x \in \Gamma_0$, and
that $\mathcal{P}_{\mathcal{G}_x}$ holds. Then each component of
$\mathcal{G}_x$ has only one state or only one signal. Let us assume for instance that
 $\mathcal{C}_1$ consists of states $1,\ldots,k$ and signal $A$. Then
$J_x^1$ equals
\begin{displaymath}
-\frac{H(x)}{x_A}\left(
\begin{array}{cccc}
x_{1} & x_{2} & \ldots & x_{k} \\
x_{1} & x_{2} & \ldots & x_{k} \\
\vdots & \vdots & \ddots & \vdots \\
x_{1} & x_{2} &  \ldots & x_{k} \\
\end{array}
 \right).
\end{displaymath}
The rank of $J_x^1$ is $1$, and its eigenvalues are $0$ and
$-H(x)$, which completes the proof.

\textbf{(c)} Let $x\in\G_u$ and $\theta\in\mathcal{E}_x\setminus\{0\}$.
Let $\{\textbf{e}_{ij}\}$ be a orthogonal basis set in
$\mathbb{R}^{M_1\times M_2}$. Let
$$W_{n+1}:=\1_{|V_{n+1}-V_n|=1}\big(V_{n+1}-V_n-x_n\big)=(1+T_n)(x_{n+1}-x_n).$$
We note that
\begin{align*}
  W_{n+1}\cdot \theta=0\,\,,&\quad \textrm{with probability }
  1-\frac{H(x_n)}{M_1}\,,\\
\forall i\in\Sc_1, j\in\Sc_2, \,\,\, W_{n+1}\cdot \theta= (1-x_{ij}^n)\textbf{e}_{ij}\cdot\theta \,\,, & \quad
  \textrm{with probability } \frac{x_{ij}^ny_{ij}^n}{M_1}\,.
\end{align*}
Note that $x\in\G_u$ implies that $H(x)\ne M_1$ and that, for all $(i,j)\in S_x$, $x_{ij}\ne1$.

Therefore, assuming that $x_n$ is in the neighbourhood of $x$ (for which $y_{ij}=H(x)$), Lemma \ref{basest} implies
\begin{align}
\label{var1}
  \textrm{Var}\big(W_{n+1}\cdot
  \theta\,|\,\mathcal{F}_n\big)\,\ge\,
  \max_{i,j\in \mathcal{C}_1 : i\sim j}\left\{\min\left
  (1-\frac{H(x_n)}{M_1},\frac{x_{ij}^ny_{ij}^n}{M_1}\right)
  \frac{((1-x_{ij}^n)\textbf{e}_{ij}\cdot\theta)^2}{4}\right\}\\
  \label{var1.1}
  \, \ge \,\mathsf{Cst}(x)\,>\,0,
\end{align}  where we use $\theta\in\mathcal{E}_x$ in the penultimate inequality. This completes the proof.
$\Box$

If $M_1=M_2$, let us define the set of equilibria $\Gamma_{\textnormal{sig}}$ as follows:
$x\in \Gamma_{\textnormal{sig}}$ if and only if there exists a
bijective map $\phi$ from $\mathcal{S}_1$ to $\mathcal{S}_2$ such
that
\begin{align*} x_{ij}&:=\left\{\begin{array}{ll} \frac{1}{M_1} & \quad \textrm{if } j=\phi(i),\\
0 & \quad\textrm{otherwise}.\end{array} \right. \end{align*}

Equilibria in $\Gamma_{\textnormal{sig}}$ correspond to a perfectly efficient signaling
system, in the sense that they bear no synonyms or informational bottlenecks.
It is easy to check from the calculation in the proof of (b) of
Proposition \ref{Prop:unstable}  that the Jacobian
matrix at such an $x$ has the only  eigenvalue
$-H(x)$, which implies asymptotic stability.
\begin{corollary}\label{signalingsys}
Suppose the game has the same number of states and signals, i.e.
$M_1=M_2$ and that $x\in \Gamma_{\textnormal{sig}}$. Then $x$ is
asymptotically stable for the mean-field ODE.
\end{corollary}
\section{Links with static equilibrium analysis}
\label{SEC:4:11}
We present in this section some results on the characterization of \emph{evolutionarily stable strategies} (ESS) and \emph{neutrally stable
strategies} (NSS) for signaling games. Note that Taylor and Jonker
(1978), and Zeeman (1979) propose, in a general setting,  conditions under which an evolutionarily stable strategy is indeed a (dynamically) stable equilibrium. However the condition is quite strong, and only applies to special cases of the signaling game.

We show here an equivalence between this
static context and the underlying
reinforcement learning dynamics, i.e. that
\begin{description}
\item[(1)] the set of ESS  matches the set of asymptotically stable
equilibria of the mean-field ODE;
\item[(2)] the set of neutrally stable strategies matches the set of linearly stable equilibria of the mean-field ODE;
\item[(3)] the set of Nash equilibria matches the set $\Lambda$ where $\nabla H\cdot F$ vanishes.
\end{description}
In Sections \ref{stat-set}, \ref{eq-sel} and \ref{connection}  we successively present usual notions in the static setting of signaling games, results on equilibrium selection and the connection between stability of the dynamics and the static setting of the game mentioned above.

\subsection{Static setting}
\label{stat-set}
Let
\begin{align*}
  \mathcal{P}&=\{P\in\mathbb{R}^{M_1\times
  M_2}_+:\forall\, i\in\mathcal{S}_1, \,\sum_{j\in\mathcal{S}_2} p_{ij}\,=\,1\,\},\\
    \mathcal{Q}&=\{Q\in\mathbb{R}^{M_2\times
  M_1}_+:\forall\, j\in\mathcal{S}_2, \,\sum_{i\in\mathcal{S}_1} q_{ji}\,=\,1\,\}.
\end{align*}
A Sender's strategy can be represented by a $M_1\times M_2$ matrix
$P\in\mathcal{P}$ in the sense that if he sees state $i$, he chooses
signal $j$ with probability $p_{ij}$.  Also we can represent
Receiver's strategy by a $M_2\times M_1$ matrix $Q\in\mathcal{Q}$, i.e.
if he sees signal $j$ he chooses act $i$
with probability $q_{ji}$. The payoff function is
$$\textrm{Payoff}(P,Q)=\sum_{i\in\mathcal{S}_1,j\in\mathcal{S}_2}p_{ij}q_{ji}=tr(PQ).$$

As in the real world, where somebody can be sender at one time and
 receiver at another, the authors assume that each individual
can be Sender or Receiver with equal probability at each time, which symmetrizes the game. Thus the (mixed) strategy of any of the two players of the symmetrized
game is a pair of Sender and Receiver matrices $(P,Q)$, and the payoff
function $\Upsilon$ can be written as
$$\Upsilon[(P,Q),(P',Q')]=\Upsilon[(P',Q'),(P,Q)]=\frac{1}{2}tr(PQ')+\frac{1}{2}tr(P'Q).$$
\subsection{Equilibrium selection}\label{SEC:4:12}
\label{eq-sel}
\begin{definition}
A strategy $(P,Q)\in\mathcal{P}\times\mathcal{Q}$ is called a
  $\textrm{Nash strategy}$ if
  $$\Upsilon[(P,Q),(P,Q)]\ge \Upsilon[(P',Q'),(P,Q)], \,\forall\, (P',Q')\in\mathcal{P}\times\mathcal{Q}.$$
\end{definition}
There are uncountably many Nash strategies in signaling games. Let us recall the following notions of evolutionary and neutrally stable equilibria, which enable one to distinguish the relevant limiting strategies of the game; note that the notions are purely static here.
\begin{definition}
  A strategy $(P,Q)\in\mathcal{P}\times\mathcal{Q}$ is
  $\textnormal{evolutionarily stable}$ if\\
    $\textnormal{(i)}\,$ it is a Nash strategy, and\\
    $\textnormal{(ii)}$ $\Upsilon[(P,Q),(P',Q')]>\Upsilon[(P',Q'),(P',Q')]$ for all
    $(P',Q')\ne(P,Q)$.
  \end{definition}
\begin{definition}
  A strategy $(P,Q)\in\mathcal{P}\times\mathcal{Q}$ is
  $\textnormal{neutrally stable}$ if
   \\
    \textnormal{(i)} $\,$ it is a Nash strategy, and\\
    \textnormal{(ii)} if $\Upsilon[(P,Q),(P,Q)]=\Upsilon[(P',Q'),(P,Q)]$ for some
    $(P',Q')\in\mathcal{P}\times\mathcal{Q}$, then $$\Upsilon[(P,Q),(P',Q')]\ge
    \Upsilon[(P',Q'),(P',Q')].$$
  \end{definition}
The following Propositions \ref{TN}, \ref{Paw} and \ref{TN1} characterize ESSs and NSSs
in the signaling game.
\begin{proposition}[Trapa and Nowak (2000)]\label{TN}
Let $(P,Q)\in\mathcal{P}\times\mathcal{Q}$ be such that neither
$P$ nor $Q$ contains a column that consists entirely of zeros. Then
$(P,Q)$ is a Nash strategy if and only if there exist positive numbers
$p_1,\ldots,p_n$ and $q_1,\ldots,q_m$ such that
\begin{align*}
 \textnormal{(1)} &\textrm{ for each $j$, the $j$-th column of $P$ has its entries
  drawn from $\{0,p_j\}$},\\
 \textnormal{(2)}  &\textrm{ for each $i$, the $i$-th column of Q has its entries
  drawn from $\{0,q_i\}$}\\
 \textnormal{(3)} &\textrm{ for all $i$, $j$, $p_{ij}\ne0$ if and only if
  $q_{ji}\ne0$.}
\end{align*}
\end{proposition}
\noindent\textbf{Remark.}
 \textnormal{The assumption that neither
$P$ nor $Q$ contains a column consisting entirely of zeros
corresponds to the requirement that no signal or act falls out of use.}

\begin{proposition}[Pawlowitsch (2007)]\label{Paw}
  Let $(P,Q)\in\mathcal{P}\times\mathcal{Q}$ be a Nash strategy. Then
  $(P,Q)$ is a neutrally stable strategy if and only if
  \begin{align*}
    \textnormal{(1)}&  \textrm{ at least one of the two matrices, P or Q, has no
    zero column, and}\\
    \textnormal{(2)}&  \textrm{ if $P$ or $Q$ contains a column with more than one
    positive element, then } \\& \textrm{ all the elements in this column take values in $\{0,1\}$}.
    \end{align*}
\end{proposition}
\begin{proposition}[Trapa and Nowak (2000)]\label{TN1}
$(P,Q)\in\mathcal{P}\times\mathcal{Q}$ is an evolutionarily stable
strategy if and only if $M_1=M_2$, $P$ is a permutation matrix and
$Q=P^T$.\end{proposition}

\subsection{Connection}
\label{connection}
Let us now define a map between the static  and dynamic learning models, in order to emphasize the correspondence between stability in either settings.

Let $s_1$ be a bijective map from $\mathcal{S}_1$ to
$\{1,\ldots,M_1\}$ and $s_2$ be a bijective map from $\mathcal{S}_2$
to $\{1,\ldots,M_2\}$. Let us define map $\Psi$ from
$\Delta\setminus\partial\Delta$ to $\mathcal{P}\times\mathcal{Q}:\,
\Psi(x)=(P,Q) $ where
\begin{align*}
 p_{s_1(i)s_2(j)}=\frac{x_{ij}}{x_i}, \,\,
q_{s_2(j)s_1(i)}=\frac{x_{ij}}{x_j},\qquad \forall\,i\in
\mathcal{S}_1, \, j\in\mathcal{S}_2.
\end{align*}

\begin{proposition} Let
\begin{align*}\G:=\Big\{(P,Q)\in\mathcal{P}\times\mathcal{Q}:&
\textrm{ neither  $P$ nor $Q$ contains} \\ &\textrm{any column that consists entirely of zeros}\Big\},\end{align*}
and let
\begin{align*}
 \mathcal{L}_{\mathcal{P}\times\mathcal{Q}}:=&\left\{(P,Q)\in\mathcal{P}\times\mathcal{Q}
\textrm{ is a Nash strategy }
\right\}\cap\G,\\
\bar{\mathcal{L}}_{\mathcal{P}\times\mathcal{Q}}
:=&\left\{(P,Q)\in\mathcal{P}\times\mathcal{Q}
\textrm{ is a neutrally stable strategy } \right\}\cap\G,\\
\tilde{\mathcal{L}}_{\mathcal{P}\times\mathcal{Q}}:=&\left\{(P,Q)\in\mathcal{P}\times\mathcal{Q}
\textrm{ is an evolutionarily stable strategy } \right\}\cap\G.
\end{align*} Then
\begin{description}
\item[(a)]
 $\Psi((\Delta\setminus\partial\Delta)\cap\Lambda)\,=\,
   \mathcal{L}_{\mathcal{P}\times\mathcal{Q}}\quad$ and $\quad
   \Psi((\Delta\setminus\partial\Delta)\setminus \Lambda))\,\cap\,
   \mathcal{L}_{\mathcal{P}\times\mathcal{Q}}=\emptyset.$
\item[(b)]
  $\Psi(\Gamma_s)=
  \bar{\mathcal{L}}_{\mathcal{P}\times\mathcal{Q}}.$
  \item[(c)]
 $\Psi(\Gamma_{\textnormal{sig}})=
  \tilde{\mathcal{L}}_{\mathcal{P}\times\mathcal{Q}}.$
\end{description}
\end{proposition}
\begin{proof}
(a) is a direct consequence of Lemma \ref{restpoint} and Proposition
\ref{TN}.

Conversely, given $(P,Q)\in \mathcal{L}_{\mathcal{P}\times\mathcal{Q}}$, let $p_1$, $\ldots$ $p_n$ and  $q_1$, $\ldots$ $q_m$ be defined as in Proposition \ref{TN}. Note that
$$p_{s_1(i)s_2(j)}q_{s_2(j)s_1(i)}=p_{s_2(j)}q_{s_1(i)}\1_{p_{s_1(i)s_2(j)}\ne0}=p_{s_1(i)s_2(j)}q_{s_1(i)}
=p_{s_2(j)}q_{s_2(j)s_1(i)}.$$

Define
$$Z:=\sum_{(i,j)\in\Sc_{pair}}p_{s_1(i)s_2(j)}q_{s_2(j)s_1(i)}=\sum_{i\in\Sc_1}q_{s_1(i)}=\sum_{j\in\Sc_2}p_{s_2(j)},$$
and let
\begin{align*}
x_{ij}&:=\frac{p_{s_1(i)s_2(j)}q_{s_2(j)s_1(i)}}{Z}, \,\,\,\,i\in\Sc_1, j\in\Sc_2;\\
x_i&:=\frac{q_i}{\sum_{k\in\Sc_1}q_k}=\sum_{j\in\Sc_2}x_{ij}, \,\,\,\,i\in\Sc_1;\\
x_j&:=\frac{p_j}{\sum_{k\in\Sc_2}p_k}=\sum_{i\in\Sc_1}x_{ij}, \,\,\,\,j\in\Sc_2.\\
\end{align*}

Then $(P,Q)=\Psi(x)$, and $x\in\Delta\setminus\partial\Delta)\cap\Lambda$ is again a direct consequence of  Lemma \ref{restpoint} and Proposition \ref{TN}.

Let us now prove $(b)$, and assume  $(P,Q)=\Psi(x)$: if one column of $P$ (resp. $Q$) has more than one element, this corresponds to an informational bottleneck (resp. synonym). Hence Condition (2) in Proposition \ref{Paw} means that a state(act)-signal correspondence cannot be associated both to a synonym
and an informational bottleneck, which translates into $(P)_{\Gc_x}$. (c) follows from Corollary \ref{signalingsys} and Proposition \ref{TN1}.
\end{proof}

\section{Convergence  with positive probability to stable configurations}
\label{sec:>0}
Let us prove the more general
\begin{theorem}
\label{posprob2}
Let $q\in\G_s$, and let $\Nc(q)$  neighbourhood of $q$ in $\De$. Then, with positive probability,
\begin{description}
 \item[(a)] $x_n\to x\in\Nc(q)$ s.t. $\mathcal{G}_x=\mathcal{G}_q$.
\item[(b)] $V(\infty,i,j)=\infty\,\iff\,\{i,j\}$ is an edge of $\mathcal{G}$.
\end{description}
\end{theorem}

Theorem \ref{posprob2} is an obvious consequence of the following Proposition
\ref{prop:posconv}. Given $\Gc=(\Sc,\sim)$, assume that $(P)_{\Gc}$ holds: then each connected component of $\Gc$ contains either only a single state or a single signal.
Let $\pi:=\pi_\Gc : \mathcal{S}\longrightarrow\mathcal{R}_q^1$ be the function mapping $i\in\mathcal{S}$ to the single state/signal in the same connected component as $i$, with the convention that we choose the state if the component consists only of one state and one signal.

 For all $i\in\Sc$, $n\in\N$ and $\e>0$, let us define
\begin{align*}
\alpha_{i}^n&:=x_{i}^n/x_{\pi(i)}^n\\
H_n^{1}&:=\bigcap_{i\in\mathcal{S},i=\pi(i)}\{\,
    \,V_{i}^n\ge 2\epsilon n\}, \\
 H_n^{2}&:=\,\,\,\,\,\,\bigcap_{i\in\mathcal{S}}\,\,\,
\{\alpha_{i}^n\ge\epsilon\}, \\
    H_n^{3}&:=\bigcap_{i,j\in\Sc, \pi(i)\ne\pi(j)}\{\,V_{ij}^n\le \sqrt{n}\,\}.
  \end{align*}

\begin{proposition}
\label{prop:posconv}
 Let $\Gc$ be such that $(P)_\Gc$ holds, and let $\pi:=\pi_\Gc$. For all $\epsilon\in(0,1/M_1)$, if $H_n^{1}$, $H_n^{2}$ and $H_n^{3}$ hold, and $n\ge\Cst(\e,M_1,M_2)$, then, with lower bounded probability (only depending on $\e$, $M_1$ and $M_2$), for all $i$, $j$ $\in\Sc$, $k\ge n$,
  \begin{align}
    &V_{ij}^{\infty}=V_{ij}^n,\textrm{ when } \pi(i)\ne\pi(j);\\
    &\alpha_{i}^k/\alpha_{i}^n\in (1-\epsilon,1+\epsilon);\\
    &V_{i}^k\ge \epsilon k, \textrm{ when } \pi(i)=i.
  \end{align}
\end{proposition}
In the remainder of this section, we fix the graph $(\Gc,\sim)$ (and thus $\pi=\pi_\Gc$) and $\epsilon>0$. The proof consists of the following Lemmas \ref{claim4}--\ref{claim6}.

  Let, for all $i$, $j$ $\in\Sc$, $n\in\N$,
  \begin{align*}
    \tau_n^{1,i,j}&:=\inf\{k\ge n\,:\, V_{ij}^k\ne V_{ij}^n\};\\
    \tau_n^{2,i}&:=\inf\{k\ge n\,:\,
    \alpha_i^k/\alpha_i^n\notin (1-\epsilon,1+\epsilon)\};\\
    \tau_n^{3,i}&:=\inf\{k\ge n\,:\, V_{i}^k<{\epsilon}
    k\},
\end{align*}
and let
$$\tau_{n}^1:=\inf_{i,j\in\Sc, \pi(i)\ne\pi(j)}\tau_n^{1,i,j},\,\,\,
\tau_{n}^2:=\inf_{i\in\Sc}\tau_n^{2,i},\,\,\,
\tau_{n}^3:=\inf_{i\in\Sc, \pi(i)=i}\tau_n^{3,i},\,\,\,\tau_n:=\tau_n^1\wedge\tau_n^2\wedge\tau_n^3.$$

\begin{lemma}
  \label{claim4}
 If  $n\ge\Cst(\e,M_1,M_2)$, then
  \begin{align*}
  \mathbb{P}\Big(\, \tau_n^1>\tau_n^2\land \tau_n^3\,|\,\mathcal{F}_n,
  H_n^1,H_n^2,H_n^3\, \Big)
  \,\ge\,\exp\left(-2\epsilon^{-4}M_1M_2\right).
  \end{align*}
\end{lemma}
\begin{proof}
Assuming $n\ge\Cst(\e,M_1,M_2)$,
\begin{align*}
  \mathbb{P}\Big(\,\tau_n^1>\tau_n^2\land
  \tau_n^3\,|\,\mathcal{F}_n,H_n^1,
  H_n^2,H_n^3\Big)&\ge\prod_{k\ge
  n}
  \left(1-\sum_{\pi(i)\ne\pi(j)}\frac{(V_{ij}^n)^2}{V_i^kV_j^k}\right)\\
&\ge\exp\left(-\frac{3}{2}\sum_{\pi(i)\ne\pi(j), k\ge n}\frac{n}{\e^4k^2}\right)\\ &\ge\exp\left(-2\epsilon^{-4}M_1M_2\right).
\end{align*}
\end{proof}

\begin{lemma}
  \label{claim5}
 If  $n\ge\Cst(\e,M_1,M_2)$ then, for all $i\in\Sc$,
  \begin{align*}
    \mathbb{P}\Big(\,\tau_n^{2,i}>\tau_n^1\land \tau_n^3\,|\,\mathcal{F}_n,
  H_n^1,H_n^2,H_n^3\,\Big)\ge1-2\exp(-\Cst(\epsilon)n).
  \end{align*}
\end{lemma}
\begin{proof}
Fix $i\in\Sc$, $n\in\N$, and assume w.l.o.g.  that $\pi(i)\ne i$.

Let, for all $j\in\Sc$ and $k\ge n$,
$$\hV_j^k:=\hV_j^n+\sum_{l\sim j}(V_{jl}^k-V_{jl}^n),$$
which is equal to $V_j^k$ as long as $k<\tau_n^1$.

Let, for all $k\ge n$,
$$W_k:=\log\frac{\hV_i^k}{\hV_{\pi(i)}^k},$$
and let us consider the Doob decomposition of $(W_k)_{k\ge n}$:
\begin{align*}
W_k&=W_n+\De_k+\Psi_k\\
\De_k&:=\sum_{j\ge n+1}\Es(W_j-W_{j-1}|\F_{j-1}).
\end{align*}

Assume that $H_n^1$, $H_n^2$ and $H_n^3$ hold, and that $k<\tau_n$: then
\begin{align*}
&|\De_{k+1}-\De_k|=
  \left|\,\mathbb{E}\left[\,W_{k+1}-W_k\,|\,\mathcal{F}_k\right]\,\right|\,\\
  &=\,\left|\frac{1}{M_1}
  \left(\frac{V_{i\pi(i)}^k}{V_i^k}\right)^2\frac{1}{V_{\pi(i)}^k}(1+\Box((V_i^k)^{-1})
-\frac{1}{M_1}\sum_{j\sim\pi(i)}
  \frac{V_{\pi(i)j}^k}{V_j^k}\frac{V_{\pi(i)j}^k}{(V_{\pi(i)})^2}(1+\Box((V_{\pi(i)}^k)^{-1})\right|\\
&=\frac{1}{M_1V_{\pi(i)}^k}\left|1+k^{-1/2}\Box(\Cst(\e,M_1,M_2)-\sum_{j\sim\pi(i)}\left(1+k^{-1/2}\Box(\Cst(\e,M_1,M_2)\right)
\frac{V_{\pi(i)j}^k}{V_{\pi(i)}^k}\right|\\
&=k^{-3/2}\Box\left(\Cst(\e,M_1,M_2)\right),
\end{align*}
where we use that, for all $j\sim\pi(i)$,
\begin{align}
\label{estsm}
\left|\frac{V_{\pi(i)j}^k}{V_j^k}-1\right|,\,\,\,
\left|\sum_{l\sim\pi(i)} \frac{V_{\pi(i)l}^k}{V_{\pi(i)}^k}-1\right|\,\,\, \le k^{-1/2}\Cst(\e,M_1,M_2).
\end{align}

Therefore, for all $k\ge n$,
$$|\De_k|\le n^{-1/2}\Cst(\e,M_1,M_2).$$

Let us now estimate the martingale increment: $|\Psi_{k+1}-\Psi_k|\le\Cst(\e)k^{-1}$ (since
$|W_{k+1}-W_k|\le\Cst(\e)k^{-1}$), so that Lemma \ref{martest} implies
$$\Pb\left(\sup_{k\ge n}|\Psi_k-\Psi_n|\1_{\{k\le\tau_n\}}\ge\e/2\right)\ge1-2\exp(-\Cst(\epsilon)n),$$
which completes the proof.
\end{proof}

\begin{lemma}
  \label{claim6}
If $\e\in(0,1/M_1)$ and $n\ge\Cst(\e,M_1,M_2)$ then, for all $i\in\Sc$ such that $\pi(i)=i$,
  \begin{align*}
  \mathbb{P}\Big(\,\tau_n^{3,i}>\tau_n^1\land \tau_n^2\,\big|\,
  \mathcal{F}_n,H_n^1,H_n^2\,\Big)>1-2\exp(-\Cst(\e)n).
  \end{align*}
\end{lemma}
\begin{proof}
Let $n\in\N$, assume that $H_n^1$, $H_n^2$ and $H_n^3$ hold, and fix $i\in\Sc$ such that $\pi(i)=i$. Let us consider the Doob decomposition of $(V_i^k)_{k\ge n}$:
\begin{align*}
V_i^k&:=V_i^n+\Phi_k+\Xi_k\\
\Phi_k&:=\sum_{j\ge n+1}\Es\left(V_i^j-V_i^{j-1}\,|\,\F_{j-1}\right).
\end{align*}

Now, for all $\eta>0$, if $n\ge\Cst(\eta,\e)$ and $k<\tau_n$, \eqref{estsm} implies
$$\Phi_{k+1}-\Phi_k= \Es\left(V_i^{k+1}-V_i^{k}\,|\,\F_{k}\right)
\ge\frac{1}{M_1}\sum_{j\sim i}\frac{(V_{ij}^k)^2}{V_i^kV_j^k}\ge\frac{1}{M_1}-\eta.$$

Let us now estimate the martingale increment: let, for all $p\ge n$,
$$\chi_p:=\sum_{k=n}^{p-1}\frac{\Xi_{k+1}-\Xi_k}{k}.$$
Then, for all $p\ge n$,
$$\Xi_p=\sum_{n\le k\le p-1}(\chi_{k+1}-\chi_k)k
=-\sum_{n\le k\le p-1}\chi_k+(p-1)\chi_p.$$

This implies, using Lemma \ref{martest} (and $|\Xi_{k+1}-\Xi_k|\le1$ for all $k\ge n$) that, for all $\e>0$,
\begin{align*}
&\Pb\left(\forall k\ge n,\,\, V_i^k\ge(2\e-\eta)n+(k-n)(1/M_1-\eta)\,|\,\F_n\right)\\ &\ge
\Pb\left(\sup_{p\ge n}\left|\frac{\Xi_p}{p}\right|\le\eta\,|\,\F_n\right)
\ge\Pb\left(\sup_{k\ge n}\left|\chi_k\right|\le\frac{\eta}{2}\,|\,\F_n\right)
\ge1-2\exp(-\Cst(\eta)n);
\end{align*}
we choose $\eta=\min(\e/2,1/M_1-\e)$, which completes the proof.

\end{proof}

\begin{lemma}
\label{martest}
Let $(\g_k)_{k\in\N}$ be a deterministic sequence of positive reals, let $\Gf:=(\Gg_n)_{n\in\N}$ be a filtration, and let $(M_n)_{n\in\N}$ be a $\Gf$-adapted martingale such that $|M_{n+1}-M_n|\le\g_n$ for all $n\in\N$. Then, for all $n\in\N$,
$$\Pb\left(\sup_{k\ge n}(M_k-M_n)\ge\la\,|\,\Gg_n\right)\le\exp\left(-\frac{\la^2}{2\sum_{k\ge n}\g_k^2}\right).$$
\end{lemma}
\bp
Let, for all $\tet\in\R$,
$$Z_n(\tet):=\exp\left(\tet M_n-\frac{\tet^2}{2}\sum_{k=1}^n\g_k^2\right).$$
Then $(Z_n(\tet))_{n\in\N}$ is supermartingale, so that
\begin{align*}
&\Pb\left(\sup_{k\ge n}(M_k-M_n)\ge\la\,|\,\Gg_n\right)\\
&\le\Pb\left(\sup_{k\ge n}Z_k(\tet)\ge Z_n(\tet)\exp\left(\la\tet-\frac{\tet^2}{2}\sum_{k=1}^n\g_k^2\right)\,|\,\Gg_n\right)\\
&\le\exp\left(\frac{\tet^2}{2}\sum_{k=1}^n\g_k^2-\la\tet\right)=\exp\left(-\frac{\la^2}{2\sum_{k\ge n}\g_k^2}\right)
\end{align*}
if we choose $\tet:=\la/\sum_{k\ge n}\g_k^2$.
\ep

\nocite{*}
\bibliographystyle{plain}
\bibliography{Thesis2010}

\end{document}